\documentclass[10pt,a4paper]{article}
\usepackage[numbers,sort&compress]{natbib}
\usepackage[T1]{fontenc}
\usepackage[utf8]{inputenc}
\usepackage{authblk}
\usepackage{amsmath,amssymb,amsthm,esint,bm}
\usepackage{mathrsfs}
\usepackage{bookmark}
\usepackage{amsmath}
\allowdisplaybreaks[3]

\newtheorem{definition}{Definition}[section]
\newtheorem{theorem}[definition]{Theorem}
\newtheorem{lemma}[definition]{Lemma}

\newtheorem{corollary}[definition]{Corollary}
\theoremstyle{remark}
\newtheorem{remark}[definition]{Remark}
\numberwithin{equation}{section}

\newcommand{\R}{\mathbb{R}}

\allowdisplaybreaks[4]

\title{Pointwise and Oscillation Estimates via Riesz Potentials for Mixed Local and Nonlocal Parabolic Equations}

\author[a]{Lingwei Ma}
\author[b]{Qi Xiong}
\author[c]{Zhenqiu Zhang\thanks{Corresponding author.}}

\affil[a]{School of Mathematical Sciences, Nankai University, Tianjin 300071, P.R. China}
\affil[b]{School of Mathematics, Southwest Jiaotong University, Chengdu 610031, Sichuan, P.R. China}
\affil[c]{School of Mathematical Sciences and LPMC, Nankai University, Tianjin 300071, P.R. China}

\date{\today}

\usepackage{hyperref}
\begin{document}
\maketitle
\footnotetext[1]{E-mail: mlw1103@163.com (L. Ma), xq@swjtu.edu.cn (Q. Xiong), zqzhang@nankai.edu.cn (Z. Zhang).}

\begin{abstract}
We establish a class of pointwise estimates for weak solutions to mixed local and nonlocal parabolic equations involving measure data and merely measurable coefficients via caloric Riesz potentials. Such estimates effectively bound the sizes and oscillations of weak solutions, respectively.
The proof relies on  demonstrating a new local H\"{o}lder estimate with an optimal $L^q$-Tail for weak solutions to the corresponding homogeneous problem, which remarkably extends the $L^\infty$-Tail in previous work. It is worth mentioning that our main results capture both local and nonlocal features of the double phase parabolic equations and, more importantly, remain valid for SOLA (Solutions Obtained by Limit of Approximations).
\\



\end{abstract}


\section{Introduction}\label{section1}

In this paper, we consider the following measure data parabolic problem, which involves a mixed diffusion characterized by the coexistence of local and nonlocal operators
\begin{equation}\label{eq1}
\partial_t u-\operatorname{div}\left(\mathcal{A}(t,x, Du)\right) +\mathcal{L} u = \mu ~~\text{in}~~\Omega_T.
\end{equation}
Here $\Omega_T=I\times \Omega$ denotes the space-time cylinder formed by
the time interval $I=(0,T)$ and the bounded domain $\Omega\subset\mathbb{R}^d$ with $d \geq 2$, and $\mu\in\mathcal M (\mathbb R^{d+1})$ is a general Radon measures with finite total mass on
$\mathbb R^{d+1}$.
The vector field $\mathcal{A}(t,x,\xi): \Omega_T\times\mathbb{R}^{ d}\rightarrow\mathbb{R}^{d}$ is a Carath\'{e}odory function satisfying the following growth and ellipticity conditions
\begin{equation}\label{growth1}
  \left|\mathcal{A}(t,x,\xi)\right|\leq\Lambda|\xi|,
\end{equation}
\begin{equation}\label{ellipticity}
\left( \mathcal{A}(t,x,\eta)-\mathcal{A}(t,x,\xi)\right)\cdot(\eta-\xi)\geq\Lambda^{-1}|\eta-\xi|^{2}
\end{equation}
for almost all $(t,x)\in\Omega_T$ and every $\xi,\,\eta\in\mathbb{R}^{d}$ with the structure constant $\Lambda\geq 1$. The nonlocal operator $\mathcal{L}$ is defined by
\begin{equation}\label{eq:L}
\mathcal{L}  u(t,x) = \text{P.V.}\int_{\mathbb{R}^d}(u(t,x)-u(t,y)){K}(t,x,y) \operatorname{d}\!y,
\end{equation}
where the kernel $ {K}: \mathbb{R} \times \mathbb{R}^d \times \mathbb{R}^d \to [0, \infty)$ is assumed to be measurable, and fulfilling the following ellipticity condition
\begin{equation}\label{ellip}
  \Lambda^{-1}|x - y|^{-d-2s} \leq  {K}(t,x, y) \leq \Lambda|x - y|^{-d-2s}
\end{equation}
for a.e. $(t, x, y) \in I \times \mathbb{R}^d \times \mathbb{R}^d$ with the fractional order $s \in (0,1)$.
The prototype of such mixed operators is $-\Delta+(-\Delta)^s$, which arises naturally from the stochastic processes that encompass both classical random walk and L\'{e}vy flight. This type of operator can be used
to model many important physical phenomena, such as heat transport in magnetized plasmas \cite{BdC13}, the spread of pandemics \cite{EG07}, and population dynamics \cite{DPV23}. Consequently, significant attention has been devoted to exploring the qualitative properties and regularity theory of solutions to mixed local-nonlocal elliptic and parabolic equations, please refer to \cite{BDVV22, BKK24, BVDV21, DM24, FSZ22, GK22, GK24, N23, SZ24} and references therein.

A powerful approach to studying the regularity of solutions to the linear Poisson equation $-\Delta u=\mu$ is based on an explicit representation formula via the fundamental solution. As an immediate consequence, the solution and its gradient can be controlled pointwise by Riesz potentials. Due to the behaviour of Riesz potentials in various relevant function spaces is known, such potential estimates provide a unified framework for exploring the regularity properties of solutions at every function space scale.
A natural question is whether these pointwise potential estimates are still available for general partial differential equations that lack representation formulas for solutions.

In the context of elliptic equations, the first affirmative answer can be traced back to the pioneering work of Kilpel\"{a}inen and Mal\'{y} \cite{KiMa, KiMa2}, who obtained pointwise estimates for solutions to the $p$-Laplace type equation with merely measurable coefficients
via the nonlinear Wolff potential of the nonhomogeneous term. Afterwards Trudinger and Wang \cite{TrWa} offered an alternative approach. Nearly twenty years later, Mingione and his groups \cite{DuMi2, KuMi, KM13, Min} upgraded the previous pointwise estimates to gradient as well as oscillation levels, provided the coefficients are appropriately regular.
Subsequently, these types of pointwise potential estimates have been extended to various kinds of elliptic equations and systems, we refer to \cite{BM, BS23, CiSc, DKLN24, DKM14, KuMinSi, KuMi3, MP20, MZ, MZ24, MZZ, NOS24, XZM} and references therein for details.

Within the framework of local parabolic equations, Duzaar and Mingione \cite{DuMi3} showed that the solution of the local parabolic equation
\begin{equation*}
\partial_t u-\operatorname{div}\left(\mathcal{A}(t,x, Du)\right) = \mu ~~\text{in}~~\Omega_T
\end{equation*}
satisfies the following pointwise estimate
\begin{equation*}
  |u(t_0,x_0)|\leq C \fint_{Q_{R}(t_0,x_0)} |u|  \operatorname{d}\!x\operatorname{d}\!t+C
\mathcal {I}^{\mu}_{2}(t_0,x_0;2R),
\end{equation*}
whenever the vector field $\mathcal{A}$ fulfills the conditions \eqref{growth1} and \eqref{ellipticity}, the parabolic cylinder $Q_{2R}(t_0,x_0)=(t_0-(2R)^2,t_0)\times B_{2R}(x_0)\Subset \Omega_T$, and the caloric Riesz potential $\mathcal {I}^{\mu}_{2}(t_0,x_0;2R)$ of order $2$ is defined in Definition \ref{Rieszpotential}.
In the nonlocal parabolic setting,
Nguyen, Nowak, Sire, and Weidner recently included in their paper \cite{NNSW23} a pointwise estimate for the weak solution of the fractional heat equation given by
\begin{equation*}
\partial_t u +\mathcal{L} u = \mu ~~\text{in}~~\Omega_T,
\end{equation*}
where the integral operator $\mathcal L$ is defined as in \eqref{eq:L}
possesses a symmetric kernel satisfying \eqref{ellip}.
They demonstrated the following pointwise bound for the solution in terms of a parabolic Riesz potential of order $2s$ and a tail term to control the long range interaction in space
\begin{eqnarray*}
|u(t_0,x_0)|  \leq C \left[\left(\fint_{\tilde{Q}_{R}(t_0,x_0)} |u| ^q \operatorname{d}\!x\operatorname{d}\!t \right)^{\frac{1}{q}}+  \left(\fint_{{\tilde{I}^\ominus_{R}(t_0)}}
   \operatorname{\widetilde{Tail}}(u(t);x_0,R)^q\operatorname{d}\!t\right)^{\frac{1}{q}}
+ \tilde{\mathcal {I}}^{\mu}_{2s}(t_0,x_0;2R) \right]
\end{eqnarray*}
for any $q\in(1,2]$, where $s\in(0,1)$ and the nonlocal tail term is given by
\begin{equation*}
\operatorname{\widetilde{Tail}}(u(t);x_0,R) = r^{2s} \int_{\R^d \setminus B_R(x_0)} \frac{|u(t,y)|}{|x_0-y|^{d+2s}} \operatorname{d}\! y.
\end{equation*}
In this case, the parabolic cylinder $\tilde{Q}_{2R}(t_0,x_0):=\tilde{I}^\ominus_{2R}(t_0)\times B_{2R}(x_0)=(t_0-(2R)^{2s},t_0)\times B_{2R}(x_0)\Subset \Omega_T$, then the  corresponding parabolic Riesz potential is intrinsically encoded by the size of this parabolic cylinder, defined as
\begin{align*}
\tilde{\mathcal {I}}^\mu_{2s}(t_0,x_0;2R)=\int_{0}^{2R}\frac{|\mu|(\tilde{Q}_\rho(t_0,x_0))}{\rho^{d}}\frac{\operatorname{d}\!\rho}{\rho}.
\end{align*}
For further earlier results on the potential theory for various parabolic equations, see \cite{DZ22, KM14-1, KM14-2, KM14-3, Ng23} and references therein.

To the best of our knowledge, there remains a notable absence of potential theory for solutions of parabolic equations involving local and nonlocal diffusion operators. Our primary focus is on establishing both pointwise and oscillatory potential estimates for solutions of the double phase parabolic equation \eqref{eq1} in terms of caloric Riesz potentials of measure data $\mu$ and the optimal tail term, assuming that the coefficients are only measurable.
To illustrate the main results of this paper, we start by presenting the definition of local weak solutions to the mixed local and nonlocal parabolic equation \eqref{eq1}.

\begin{definition} Let $\mu \in \mathcal{M}(\mathbb{R}^{d+1})$. A function
$$u\in L_{\rm loc}^2(I;W^{1,2}_{\rm loc}(\Omega))\cap L_{\rm loc}^\infty(I;L^2_{\rm loc}(\Omega))\cap L^2_{\rm loc}(I;{\mathcal L}^1_{2s}(\mathbb{R}^d))$$
is a local weak subsolution to \eqref{eq1}, if
\begin{eqnarray*}
		&& -\int_{t_1}^{t_2} \int_{\Omega} u\partial_t\varphi \operatorname{d}\! x \operatorname{d}\!t +\int_{\Omega}u(t_2,x) \varphi(t_2,x) \operatorname{d}\!x-\int_{\Omega}u(t_1,x) \varphi(t_1,x) \operatorname{d}\!x\\
&&+\int_{t_1}^{t_2} \int_{\Omega}\mathcal{A}(t,x,Du)\cdot D\varphi\operatorname{d}\! x \operatorname{d}\!t + \int_{t_1}^{t_2} \int_{\mathbb{R}^d} \int_{\mathbb{R}^d} (u(t,x)-u(t,y))(\varphi(t,x)-\varphi(t,y)) K(t,x,y) \operatorname{d}\!x \operatorname{d}\!y \operatorname{d}\! t \\
&\leq&\int_{t_1}^{t_2} \int_{\Omega} \varphi \operatorname{d}\!\mu .
\end{eqnarray*}
holds for all $[t_1,t_2] \subset I$ and nonnegative test functions
$\varphi\in W_{\rm loc}^{1,2}(I;L^2(\Omega))\cap L_{\rm loc}^2(I; W^{1,2}(\Omega))$
with compact spatial support contained in $\Omega$, where ${\mathcal L}^1_{2s}(\mathbb{R}^d)$ represents the nonlocal tail space defined as
	$${\mathcal L}^1_{2s}(\mathbb{R}^d):= \left \{v \in L^1_{\operatorname{loc}}(\mathbb{R}^d) \mathrel{\Big|} \int_{\mathbb{R}^d} \frac{|v(y)|}{1+|y|^{d+2s}} \operatorname{d}\!y < \infty \right \}.$$
A function $u$ is called a local weak supersolution to \eqref{eq1} if the previous formula holds for every nonpositive test function $\varphi$. If a function $u$ is both a local weak subsolution and a local weak supersolution, then it is referred to as a local weak solution to \eqref{eq1}.
\end{definition}

Here we do not need to impose the additional symmetry assumption on the kernel $K$ that
\begin{equation}\label{symm}
  {K}(t,x, y) = {K}(t,y, x)\,\,\mbox{for a.e.} \,\,(t, x, y) \in I \times \mathbb{R}^d \times \mathbb{R}^d,
\end{equation}
as in \cite{NNSW23}. Indeed, without loss of generality, we may assume that the kernel $K$ is symmetric. Otherwise,
let us rename variables to show that
\begin{eqnarray*}
   && \int_{t_1}^{t_2} \int_{\mathbb{R}^d} \int_{\mathbb{R}^d} (u(t,x)-u(t,y))(\varphi(t,x)-\varphi(t,y)) K(t,x,y) \operatorname{d}\!x \operatorname{d}\!y \operatorname{d}\! t \\
  &=& \int_{t_1}^{t_2} \int_{\mathbb{R}^d} \int_{\mathbb{R}^d} (u(t,x)-u(t,y))(\varphi(t,x)-\varphi(t,y)) K(t,y,x) \operatorname{d}\!y \operatorname{d}\!x \operatorname{d}\! t \\
   &=&  \int_{t_1}^{t_2} \int_{\mathbb{R}^d} \int_{\mathbb{R}^d} (u(t,x)-u(t,y))(\varphi(t,x)-\varphi(t,y)) \frac{K(t,x,y)+K(t,y,x)}{2} \operatorname{d}\!x \operatorname{d}\!y \operatorname{d}\! t.
\end{eqnarray*}
It is obvious that the new kernel $\bar{K}(t,x,y):=\frac{K(t,x,y)+K(t,y,x)}{2}$ is symmetric in the sense of \eqref{symm} and satisfies the ellipticity condition
$\Lambda^{-1}|x - y|^{-d-2s} \leq  {\bar{K}}(t,x, y) \leq \Lambda|x - y|^{-d-2s}$ for a.e. $(t, x, y) \in I\times \mathbb{R}^d \times \mathbb{R}^d$
by \eqref{ellip}. Therefore, we can replace the original kernel $K$ with a symmetric kernel $\bar K$.

In what follows, we denote the parabolic cylinder $Q_r(t_0,x_0):= I_r^{\ominus}(t_0) \times B_r(x_0)$, where the backward time interval $I_r^{\ominus}(t_0) := (t_0-r^{2},t_0)$.
We proceed by introducing the definition of the caloric Riesz potential of measure $\mu$, which will be utilized in our potential estimates.
\begin{definition}\label{Rieszpotential}
Let $\alpha\in (0,d+2)$, the parabolic Riesz potential $\mathcal {I}^\mu_{\alpha}(t_0,x_0;R)$ of $\mu\in \mathcal M(\mathbb R^{d+1})$ with order $\alpha$ is defined by
\begin{align*}
\mathcal {I}^\mu_{\alpha}(t_0,x_0;R)=\int_{0}^{R}\frac{|\mu|(Q_\rho(t_0,x_0))}{\rho^{d+2-\alpha}}\frac{\operatorname{d}\!\rho}{\rho}
\end{align*}
for any $(t_0,x_0)\in I\times \Omega$ such that the parabolic cylinder $Q_R(t_0,x_0)\subset I\times \Omega$.
\end{definition}

We are now in a position to state the first result of this paper, specifically, a pointwise potential estimate of solutions to the mixed local and nonlocal parabolic equation \eqref{eq1}.
\begin{theorem}\label{thm:PE}
Let $u$ be a local weak solution to \eqref{eq1} with $\mu \in \mathcal{M}(\mathbb{R}^{d+1})$ in $\Omega_T$,
then for any $q\in(1,2]$, any $R\in(0,1]$ and almost all $(t_0,x_0)\in I\times\Omega$ with $Q_{2R}(t_0,x_0) \Subset I \times \Omega$, there exists a
positive constant $C$ depending only on $d,\,s,\,\Lambda,\,q$
such that
\begin{eqnarray} \label{eq:PE}
|u(t_0,x_0)|  \leq C \left[\left(\fint_{Q_{R}(t_0,x_0)} |u| ^q \operatorname{d}\!x\operatorname{d}\!t \right)^{\frac{1}{q}}+  \left(\fint_{{I^\ominus_{R}(t_0)}}
   \operatorname{Tail}(u(t);x_0,R)^q\operatorname{d}\!t\right)^{\frac{1}{q}}
+ \mathcal {I}^{\mu}_{2}(t_0,x_0;2R) \right],
\end{eqnarray}
where the nonlocal tail term is defined as
\begin{equation*}
\operatorname{Tail}(u(t);x_0,R) = r^{2} \int_{\R^d \setminus B_R(x_0)} \frac{|u(t,y)|}{|x_0-y|^{d+2s}} \operatorname{d}\! y.
\end{equation*}
\end{theorem}

We emphasize that the pointwise potential estimate \eqref{eq:PE} captures both local and nonlocal characters inherent in the mixed operator.
Given the nonlocal nature of the fractional term in \eqref{eq1}, it is essential to include a nonlocal tail term in the derivation of the potential estimate to effectively control the solutions.
Nevertheless, the leading contributor in \eqref{eq1} is the higher order local term, thereby our potential estimate involves the caloric Riesz potential $\mathcal {I}^{\mu}_{2}(t_0,x_0;2R)$ of order $2$, and the nonlocal tail term used here differs slightly from that discussed in \cite{NNSW23} in the context of purely nonlocal parabolic equations.

The most critical ingredient in the construction of potential estimates is to establish an excess decay estimate for the solution of the
mixed problem \eqref{eq1}, as presented in Lemma \ref{lemma:OscDecu}.
The fundamental scheme to prove this decay estimate comprises two key ingredients,
a comparison estimate given in Lemma \ref{thm:comparison}, and
the following local H\"older continuity with an optimal tail for the weak solution of the homogeneous problem corresponding to \eqref{eq1}, of the type
\begin{equation}\label{eq:PDE-hom}
\partial_t v-\operatorname{div}\left(\mathcal{A}(t,x, Dv)\right) +\mathcal{L} v= 0 ~~ \text{ in } \Omega_T.
\end{equation}

\begin{theorem}\label{prop:Holderq}
 Let $v$ be a local weak solution of \eqref{eq:PDE-hom} in $\Omega_T$,
then for any $q\in(1,2]$, there exists a constant $\gamma_0\in(0,1)$ depending only on $d,\,s,\,\Lambda,\,q$ such that $v \in C^{\frac{\gamma_0}{2},\gamma_0}_{\rm loc}(\Omega_T)$ and satisfying
	\begin{eqnarray}\label{holderv}
   [ v ]_{C^{\frac{\gamma_0}{2},\gamma_0}(Q_{R}(t_0,x_0))}\leq C R^{-\gamma_0}  \left[ \fint_{Q_{2R}(t_0,x_0)} |v(t,x)|\operatorname{d}\! x \operatorname{d}\! t  +\left(\fint_{I^\ominus_{2R}(t_0)}\operatorname{Tail}(v(t);x_0,2R)^q\operatorname{d}\! t\right)^{\frac{1}{q}}
\right]
 \end{eqnarray}
for any $R\in(0,1]$ and $(t_0,x_0)\in I\times\Omega$ with $Q_{2R}(t_0,x_0)\Subset I\times\Omega$, where
\begin{equation*}
  [ u ]_{C^{\frac{\gamma_0}{2},\gamma_0}(Q_{R}(t_0,x_0))}:=\sup_{\substack{(t,x),(\tau,y)\in Q_{R}(t_0,x_0)\\ (t,x)\neq(\tau,y) } }\frac{|u(t,x)-u(\tau,y)|}{\left(|t-\tau|^{\frac{1}{2}}+|x-y|\right)^{\gamma_0}},
\end{equation*}
and the positive constant $C$ depends only on $d,\,s,\,\Lambda,\,q$.
\end{theorem}
This result might also be interesting for its own sake. In contrast to H\"older regularity of solutions to mixed local and nonlocal parabolic equations mentioned above \cite{FSZ22, GK24}, they share a common additional assumption regarding the boundedness of the nonlocal tail term with respect to the time variable. However, we find that further proving potential estimates based on the local H\"older estimate with an $L^\infty$-Tail is unfeasible, and this assumption often does not hold in the solution space associated with weak solutions. Inspired by the recent work of Kassman and Weidner \cite{KW23}, who obtained H\"older regularity of solutions to nonlocal parabolic equations without an a priori bounded assumption on the nonlocal tail term, here we improve the tail term in the previous H\"older estimates of solutions to mixed problems with respect to the time variable from $L^\infty$ to $L^q$ for any $q\in(1,2]$. Since the weak solution $v$ belonging to $L^2_{\rm loc}(I;{\mathcal L}^1_{2s}(\mathbb{R}^d))$ apparently guarantees the nonlocal tail term $\operatorname{Tail}(v(t);x_0,2R)\in L^q_{\rm loc}(I)$ for any $q\in(1,2]$, our result does not impose any unnatural restriction on solutions. When $q=1$ in \eqref{holderv}, analogous to the discussion in \cite{KW23}, there exists a counterexample showing that the solution $v$ is not H\"older continuous, thus justifying the optimality of the nonlocal tail term in the H\"older continuity estimate \eqref{holderv}.

The pointwise estimate \eqref{eq:PE} is only able to control the size of the solution, another novel contribution here is the following oscillation estimate, which allows to express the H\"{o}lder continuity of the solution to the mixed local and nonlocal parabolic equation \eqref{eq1} in terms of the caloric Riesz potential of order $2-\gamma$.
\begin{theorem}\label{thm:OPE}
Let $u$ be a local weak solution to \eqref{eq1} with $\mu \in \mathcal{M}(\mathbb{R}^{d+1})$ in $\Omega_T$.
For any $q\in(1,2]$ and $\tilde{\gamma}\in[0,\gamma_0)$, any $R\in(0,1]$ and $(t_0,x_0)\in I\times\Omega$ with $Q_{2R}(t_0,x_0) \Subset I \times \Omega$, there holds that
\begin{eqnarray} \label{eq:OPE}
&&|u(t,x)-u(\tau,y)| \nonumber\\
 &\leq& C
\left[\left(\fint_{Q_{R}(t_0,x_0)} |u| ^q \operatorname{d}\!x\operatorname{d}\!t \right)^{\frac{1}{q}}+  \left(\fint_{{I^\ominus_{R}(t_0)}}
   \operatorname{Tail}(u(t);x_0,R)^q\operatorname{d}\!t\right)^{\frac{1}{q}}
 \right]\left(\frac{|t-\tau|^{\frac{1}{2}}+|x-y|}{R}\right)^\gamma \nonumber\\
&&+C\left[\mathcal {I}^{\mu}_{2-\gamma}(t,x;R)+\mathcal {I}^{\mu}_{2-\gamma}(\tau,y;R)\right]\left(|t-\tau|^{\frac{1}{2}}+|x-y|\right)^\gamma
\end{eqnarray}
for almost everywhere $(t,x),\,(\tau,y)\in Q_{\frac{R}{8}}(t_0,x_0)$ and every $\gamma\in[0,\tilde{\gamma}]$,
where the positive constant $C$ depends only on $d,\,s,\,\Lambda,\,q,\,\tilde{\gamma}$, and the constant $\gamma_0$ is given by Theorem \ref{prop:Holderq}\,.
\end{theorem}

\begin{remark}
The existence theory and compactness properties of a class of very weak solutions obtained by limit of approximations (SOLA) to the mixed problem \eqref{eq1} have been developed in detail in \cite{BKK24}.
It is worth mentioning that our main theorems still hold for SOLA when $q\in(1,\frac{n+2}{n+1})$ by utilizing compactness results through a standard approximation process.
\end{remark}

The remainder of this paper is organized as follows. Section \ref{section2} presents some preliminary results and establishes the Caccioppoli inequality. Section \ref{section3} is devoted to proving the local H\"older regularity stated in Theorem \ref{prop:Holderq} based on the De Giorgi-Nash-Moser theory. In section \ref{section4}, we derive the relevant comparison and excess decay estimates.
Finally, in the last section, we complete the proof of Theorem \ref{thm:PE} and Theorem \ref{thm:OPE} regarding pointwise and oscillation potential estimates.

\section{Preliminaries}\label{section2}

In this section, we first collect some useful auxiliary results and further establish a series of Caccioppoli-type inequalities for the homogeneous mixed local and nonlocal parabolic equation \eqref{eq:PDE-hom}. These results are pivotal for the proof of our main theorems.
\subsection{Auxiliary results}

Let us begin with the following iterative lemma, two technical lemmas concerning the geometric convergence of sequences of numbers, together with a poincar\'{e}-type inequality. These are essential components for carrying out the De Giorgi theory to prove the local H\"{o}lder continuity Theorem \ref{prop:Holderq}.

\begin{lemma}\label{iterate3} {\rm(cf. \cite[Lemma 6.1]{Giu})}
Let $f$ be a bounded nonnegative function defined in the interval $[r,R]$.
Assume that for $ r\leq \rho_1<\rho_2\leq R$ we have
\begin{equation*}
  f(\rho_1)\leq \vartheta f(\rho_2)+\frac{c_1}{(\rho_2-\rho_1)^{\alpha}}+\frac{c_2}{(\rho_2-\rho_1)^{\beta}}+c_3,
\end{equation*}
where $c_1,\,c_2,\,c_3\geq 0$, $\alpha>\beta>0$, and $0\leq\vartheta<1$. Then there exists a positive constant $C=C(\alpha,\vartheta)$
such that
\begin{equation*}
  f(r)\leq C\left[\frac{c_1}{(R-r)^{\alpha}}+\frac{c_2}{(R-r)^{\beta}}+c_3\right].
\end{equation*}
\end{lemma}

\begin{lemma}\label{iterate1} {\rm(cf. \cite[Chapter 1, Lemma 4.1]{Dib93})}
Let $\{A_i\}$ be a sequence of positive numbers such that
$$A_{i+1}\leq Cb^iA_i^{1+\beta},$$
where $C,\,b>1$ and $\beta>0$ are given numbers. If $A_0\leq C^{-\frac{1}{\beta}}b^{-\frac{1}{\beta^2}}$, then $\{A_i\}$ converges to zero as $i\rightarrow\infty$.
\end{lemma}

\begin{lemma}\label{iterate2} {\rm(cf. \cite[Chapter 1, Lemma 4.2]{Dib93})}
Let $\{A_i\}$ and $\{B_i\}$ be sequences of positive numbers, satisfying the recursive inequalities
\begin{equation*}
  \left\{\begin{array}{r@{\ \ }c@{\ \ }ll}
  A_{i+1}&\leq& Cb^i(A_i^{1+\beta}+A_i^{\beta}B_i^{1+\alpha}),\\
  B_{i+1}&\leq& Cb^i(A_i+B_i^{1+\alpha}),
  \end{array}\right.
\end{equation*}
where $C,\,b>1$ and $\alpha,\,\beta>0$ are given numbers.  If
$$A_0+B_0^{1+\alpha}\leq (2C)^{-\frac{1+\alpha}{\sigma}}b^{-\frac{1+\alpha}{\sigma^2}}$$
with $\sigma=\min\{\alpha,\,\beta\}$, then
$\{A_i\}$ and $\{B_i\}$ tend to zero as $i\rightarrow\infty$.
\end{lemma}

\begin{lemma}\label{auxilem1}{\rm(cf. \cite[Chapter 2, Lemma 2.2]{DGV12})}
Let $f\in W^{1,1}(B_R)$ and $k< h$, then there exists a positive constant $C$ depending only on $d$ such that
\begin{equation*}
  (h-k)|\{x\in B_R\mid f(x)\leq k\}|\leq \frac{C R^{d+1}}{|\{x\in B_R\mid f(x)\geq h\}|}\int_{\{x\in B_R\mid k<f(x)\leq h\}}|Df|\operatorname{d}\! x.
\end{equation*}
\end{lemma}

\subsection{Caccioppoli inequality}
The set of functions satisfying the Caccioppoli inequality is typically called a De Giorgi class. In this subsection, we concentrate on the Caccioppoli-type inequalities for the homogeneous mixed local and nonlocal parabolic equation \eqref{eq:PDE-hom}.
In what follows, $C$ denotes a constant whose value may vary from line to line, and only the relevant dependencies are specified in parentheses.

\begin{lemma}
\label{lemma:Cacc}
Let $v$ be a local weak subsolution to \eqref{eq:PDE-hom} in $\Omega_T$,
then for any $R\in(0,1]$ and $(t_0,x_0)\in I\times\Omega$ with $I^\ominus_{2R}(t_0)\times B_{2R}(x_0)\Subset I\times\Omega$, there exists a positive constant $C$ depending only on $d,\,s,\,\Lambda$ such that
\begin{eqnarray}\label{eq:Cacc}
 && \sup_{t\in I_{r}^{\ominus}(t_0)}\int_{B_r(x_0)} w^2_{+}(t,x) \operatorname{d}\! x +\int_{I_{r}^{\ominus}(t_0)}\int_{B_r(x_0)} |Dw_{+}(t,x)|^2\operatorname{d}\! x\operatorname{d}\! t \nonumber\\
 &&+\int_{I_{r}^{\ominus}(t_0)}\int_{B_r(x_0)} \int_{B_r(x_0)}\frac{|w_{+}(t,x)-w_{+}(t,y)|^2}{|x-y|^{d+2s}}\operatorname{d}\! x \operatorname{d}\! y\operatorname{d}\!t \nonumber\\
 &&+ r^{-d-2s} \int_{I_{r}^{\ominus}(t_0)} \int_{B_r(x_0)} \int_{B_r(x_0)} w_{+}(t,x) w_{-}(t,y) \operatorname{d}\! x \operatorname{d}\! y \operatorname{d}\! t\nonumber\\
 &\leq&  C \left(\rho_1^{-2}\vee\rho_2^{-2}\right)\int_{I_{r + \rho_1}^{\ominus}(t_0)} \int_{B_{r+\rho_2}(x_0)} w^2_{+}(t,x) \operatorname{d}\! x \operatorname{d}\! t \nonumber\\
&&+ C\left(\frac{r+\rho_2}{\rho_2}\right)^d \rho_2^{-2} \int_{I_{r+\rho_1}^{\ominus}(t_0)} \left(\int_{B_{r+\rho_2}(x_0)} w_{+}(t,x) \operatorname{d}\!x \right) \operatorname{Tail}(w_{+}(t);x_0,r+\rho_2)\operatorname{d}\!t
\end{eqnarray}
and
\begin{eqnarray}\label{eq:Cacc2}
 && \sup_{t\in I_{r+\rho_1}^{\ominus}(t_0)}\int_{B_r(x_0)} w^2_{+}(t,x) \operatorname{d}\! x +\int_{I_{r+\rho_1}^{\ominus}(t_0)}\int_{B_r(x_0)} |Dw_{+}(t,x)|^2\operatorname{d}\! x\operatorname{d}\! t \nonumber\\
 &&+\int_{I_{r+\rho_1}^{\ominus}(t_0)}\int_{B_r(x_0)} \int_{B_r(x_0)}\frac{|w_{+}(t,x)-w_{+}(t,y)|^2}{|x-y|^{d+2s}}\operatorname{d}\! x \operatorname{d}\! y\operatorname{d}\!t \nonumber\\
 &&+ r^{-d-2s} \int_{I_{r+\rho_1}^{\ominus}(t_0)} \int_{B_r(x_0)} \int_{B_r(x_0)} w_{+}(t,x) w_{-}(t,y) \operatorname{d}\! x \operatorname{d}\! y \operatorname{d}\! t\nonumber\\
 &\leq&  \int_{B_{r+\rho_2}(x_0)} w^2_{+}(t_0-(r+\rho_1)^2,x) \operatorname{d}\! x+ C \rho_2^{-2}\int_{I_{r + \rho_1}^{\ominus}(t_0)} \int_{B_{r+\rho_2}(x_0)} w^2_{+}(t,x) \operatorname{d}\! x \operatorname{d}\! t \nonumber\\
&&+ C\left(\frac{r+\rho_2}{\rho_2}\right)^d \rho_2^{-2} \int_{I_{r+\rho_1}^{\ominus}(t_0)} \left(\int_{B_{r+\rho_2}(x_0)} w_{+}(t,x) \operatorname{d}\!x \right) \operatorname{Tail}(w_{+}(t);x_0,r+\rho_2)\operatorname{d}\!t
\end{eqnarray}
for any $0 < \rho_1\leq r \leq r+\rho_1 \leq R$ and $0 < \rho_2 \leq r \leq r+\rho_2 \leq R$ satisfying $I_{r+\rho_1}^{\ominus}(t_0) \times B_{r+\rho_2}(x_0) \subset Q_{2R}(t_0,x_0)$, where $w = v-l$ with a level $l \in \R$, the lower truncation $w_+:=\max\{v-l,0\}$, the upper truncation $w_-:=-\min\{v-l,0\}$, and $\rho_1^{-2}\vee\rho_2^{-2}:=\max\{\rho_1^{-2},\rho_2^{-2}\}$. Moreover, the Caccioppoli-type inequalities \eqref{eq:Cacc} and \eqref{eq:Cacc2} still hold for $w_-$ when $v$ is a local weak supersolution to \eqref{eq:PDE-hom} in $\Omega_T$.
\end{lemma}

\begin{proof}
The proof proceeds by testing the weak subsolution formulation with the function $\varphi(t,x) = \eta^2(t)\phi^2(x)w_+(t,x)$, where the smooth function $\eta \in C^{\infty}([t_0-(r+\rho_1)^2,t_0] ;[0,1])$ with $\eta \equiv 1$ in $[t_0-r^2,t_0]$, $\eta(t_0-(r+\rho_1)^2) = 0$ and $|\eta'| \leq C \left((r+\rho_1)^2-r^2\right)^{-1}$, as well as the cutoff function $\phi \in C_0^{\infty}(B_{r+\frac{\rho_2}{2}}(x_0);[0,1])$ with $\phi \equiv 1$ in $B_{r}(x_0)$ and $|D \phi| \leq C \rho_2^{-1}$.
Note that the test function contains $v$ which lacks the required regularity, so the first step should be to regularize the time variable. In fact, it
can be carried out by directly adapting the Steklov averages technique in a standard procedure, we omit this step here.

On the basis that there is the following weak subsolution formulation
\begin{eqnarray}\label{weaksol1}
		 0&\geq&\frac{1}{2}\int_{B_{r+\rho_2}(x_0)}w_+^2(\tau,x)\phi^2(x) \operatorname{d}\!x-\int_{t_0-(r+\rho_1)^2}^{\tau} \int_{B_{r+\rho_2}(x_0)} w_+^2(t,x)\phi^2(x)\eta(t)\eta'(t)  \operatorname{d}\! x \operatorname{d}\!t \nonumber\\
&&+\int_{t_0-(r+\rho_1)^2}^{\tau} \int_{B_{r+\rho_2}(x_0)}\mathcal{A}(t,x,Dv)\cdot D\left(\eta^2(t)\phi^2(x)w_+(t,x)\right)\operatorname{d}\! x \operatorname{d}\!t \\
&&+ \int_{t_0-(r+\rho_1)^2}^{\tau} \int_{\mathbb{R}^d} \int_{\mathbb{R}^d} (v(t,x)-v(t,y))(w_+(t,x)\phi^2(x)-w_+(t,y)\phi^2(y))\eta^2(t) K(t,x,y) \operatorname{d}\!x \operatorname{d}\!y \operatorname{d}\! t\nonumber
\end{eqnarray}
for any $\tau\in[t_0-r^2,t_0]$, we estimate the integral terms in \eqref{weaksol1} separately. The choice of $\eta$ and $\phi$ implies that
\begin{equation}\label{est1}
  \frac{1}{2}\int_{B_{r+\rho_2}(x_0)}w_+^2(\tau,x)\phi^2(x) \operatorname{d}\!x\geq\frac{1}{2}\int_{B_{r}(x_0)}w_+^2(\tau,x) \operatorname{d}\!x,
\end{equation}
and
\begin{equation}\label{est2}
  \int_{t_0-(r+\rho_1)^2}^{\tau} \int_{B_{r+\rho_2}(x_0)} w_+^2(t,x)\phi^2(x)\eta(t)\eta'(t)  \operatorname{d}\! x \operatorname{d}\!t\leq C\rho_1^{-2}\int_{I_{r+\rho_1}^\ominus(t_0)} \int_{B_{r+\rho_2}(x_0)} w_+^2(t,x) \operatorname{d}\! x \operatorname{d}\!t\,.
\end{equation}
The growth condition \eqref{growth1} implies that $\mathcal{A}(t,x,{\bf{0}})=0$,
which is then combined with the ellipticity condition \eqref{ellipticity} to yield that
\begin{eqnarray*}
  \mathcal{A}(t,x,Dv)\cdot Dw_+(t,x)  &=& \left(\mathcal{A}(t,x,Dv)-\mathcal{A}(t,x,{\bf{0}})\right)\cdot Dw_+(t,x) \\
   &=&  \left(\mathcal{A}(t,x,Dw_+)-\mathcal{A}(t,x,{\bf{0}})\right)\cdot Dw_+(t,x) \geq \Lambda^{-1}|Dw_+(t,x)|^2.
\end{eqnarray*}
Substituting the above inequality into the third term on the right side of \eqref{weaksol1} and applying the growth condition \eqref{growth1} together with Young's inequality, we derive
\begin{eqnarray}\label{est3}
   &&\int_{t_0-(r+\rho_1)^2}^{\tau} \int_{B_{r+\rho_2}(x_0)}\mathcal{A}(t,x,Dv)\cdot \left(\phi^2(x)Dw_+(t,x)+2w_+(t,x)\phi(x)D\phi(x)\right)\eta^2(t)\operatorname{d}\! x\operatorname{d}\! t\nonumber\\
   &\geq&\Lambda^{-1} \int_{t_0-(r+\rho_1)^2}^{\tau} \int_{B_{r+\rho_2}(x_0)}|Dw_+(t,x)|^2\phi^2(x)\eta^2(t)\operatorname{d}\! x\operatorname{d}\! t\nonumber\\
   &&-2\Lambda \int_{t_0-(r+\rho_1)^2}^{\tau} \int_{B_{r+\rho_2}(x_0)}|Dw_+(t,x)|w_+(t,x)\phi(x)|D\phi(x)|\eta^2(t)\operatorname{d}\! x\operatorname{d}\! t\nonumber\\
    &\geq&\frac{\Lambda^{-1}}{2} \int_{t_0-r^2}^{\tau} \int_{B_{r}(x_0)}|Dw_+(t,x)|^2\operatorname{d}\! x\operatorname{d}\! t- C(\Lambda)\rho_2^{-2}\int_{I_{r+\rho_1}^\ominus(t_0)} \int_{B_{r+\rho_2}(x_0)} w_+^2(t,x) \operatorname{d}\! x \operatorname{d}\!t\,.
\end{eqnarray}
Subsequently, using the symmetry of the kernel $K$, we split the last term on the right side of \eqref{weaksol1} by dividing the integral region as follows
\begin{eqnarray}\label{laste}
  &&\int_{t_0-(r+\rho_1)^2}^{\tau}\int_{\mathbb{R}^d} \int_{\mathbb{R}^d} (v(t,x)-v(t,y))\left(w_+(t,x)\phi^2(x)-w_+(t,y)\phi^2(y)\right)\eta^2(t) K(t,x,y) \operatorname{d}\!x \operatorname{d}\!y \operatorname{d}\! t\nonumber\\
   &=& \int_{t_0-(r+\rho_1)^2}^{\tau}\int_{B_{r+\rho_2}(x_0)} \int_{B_{r+\rho_2}(x_0)} (v(t,x)-v(t,y))\left(w_+(t,x)\phi^2(x)-w_+(t,y)\phi^2(y)\right)\eta^2(t) K(t,x,y) \operatorname{d}\!x \operatorname{d}\!y \operatorname{d}\! t  \nonumber\\
   &&+2\int_{t_0-(r+\rho_1)^2}^{\tau} \int_{B_{r+\rho_2}(x_0)}\int_{\mathbb{R}^d\setminus B_{r+\rho_2}(x_0)} (v(t,y)-v(t,x))w_+(t,y)\phi^2(y)\eta^2(t) K(t,x,y) \operatorname{d}\!x \operatorname{d}\!y \operatorname{d}\! t\nonumber\\
   &:=&I+II.
\end{eqnarray}
For any fixed $t\in (t_0-(r+\rho_1)^2,\tau)$, we denote the superlevel set
$$A_l:=\{u(\cdot,t)>l\}\cap B_{r+\rho_2}\,.$$
In the first case, if $x\in A_l$ and $y\in B_{r+\rho_2}\setminus A_l$, then
\begin{eqnarray*}
   && (v(t,x)-v(t,y))\left(w_+(t,x)\phi^2(x)-w_+(t,y)\phi^2(y)\right) \\
   &=& (w_+(t,x)+w_-(t,y))w_+(t,x)\phi^2(x)
   \geq w_+(t,x)w_-(t,y)\phi^2(x).
\end{eqnarray*}
In the second case, if $x,\,y\in A_l$, then
\begin{eqnarray*}
   &&(v(t,x)-v(t,y))\left(w_+(t,x)\phi^2(x)-w_+(t,y)\phi^2(y)\right) \\
   &=&\left(w_+(t,x)\phi(x)-w_+(t,y)\phi(y)\right)^2-
   w_+(t,x)w_+(t,y)(\phi(x)-\phi(y))^2.
\end{eqnarray*}
Otherwise, if $x,\,y\in B_{r+\rho_2}\setminus A_l$, then
$$(v(t,x)-v(t,y))\left(w_+(t,x)\phi^2(x)-w_+(t,y)\phi^2(y)\right)=0.$$
Taking into account the aforementioned estimates and combining the ellipticity condition \eqref{ellip} with the selection of $\phi$ and $r\leq1$, we estimate $I$ in the following manner
\begin{eqnarray*}
  I &\geq&\Lambda^{-1}\int_{t_0-(r+\rho_1)^2}^{\tau} \int_{B_{r+\rho_2}(x_0)} \int_{B_{r+\rho_2}(x_0)} w_+(t,x)w_-(t,y)\phi^2(x)\eta^2(t) |x - y|^{-d-2s} \operatorname{d}\!x \operatorname{d}\!y \operatorname{d}\! t\\   && +\Lambda^{-1}\int_{t_0-(r+\rho_1)^2}^{\tau}\int_{B_{r+\rho_2}(x_0)} \int_{B_{r+\rho_2}(x_0)} \left(w_+(t,x)\phi(x)-w_+(t,y)\phi(y)\right)^2\eta^2(t) |x - y|^{-d-2s} \operatorname{d}\!x \operatorname{d}\!y \operatorname{d}\! t\\
   && -\Lambda\int_{t_0-(r+\rho_1)^2}^{\tau}\int_{B_{r+\rho_2}(x_0)} \int_{B_{r+\rho_2}(x_0)} \left(w^2_+(t,x)+w^2_+(t,y)\right)|D\phi|^2\eta^2(t)|x - y|^{-d-2s+2} \operatorname{d}\!x \operatorname{d}\!y \operatorname{d}\! t\\
    &\geq& C(\Lambda,d,s)r^{-d-2s} \int_{t_0-r^2}^{\tau} \int_{B_r(x_0)} \int_{B_r(x_0)} w_{+}(t,x) w_{-}(t,y) \operatorname{d}\! x \operatorname{d}\! y \operatorname{d}\! t\\
     &&+\Lambda^{-1}\int_{t_0-r^2}^{\tau} \int_{B_{r}(x_0)} \int_{B_{r}(x_0)} \frac{\left|w_+(t,x)-w_+(t,y)\right|^2}{|x - y|^{d+2s}} \operatorname{d}\!x \operatorname{d}\!y\operatorname{d}\! t\\
    && -C(\Lambda,d,s)\rho_2^{-2}\int_{I_{r+\rho_1}^{\ominus}(t_0)} \int_{B_{r+\rho_2}(x_0)} w^2_+(t,x)\operatorname{d}\!x \operatorname{d}\! t\,.
\end{eqnarray*}
Regarding the estimate of $II$, we once again utilize the ellipticity condition \eqref{ellip} to derive
\begin{eqnarray*}
  -II &\leq& 2\Lambda\int_{I_{r+\rho_1}^{\ominus}(t_0)} \int_{B_{r+\rho_2}(x_0)} \int_{\mathbb{R}^d\setminus B_{r+\rho_2}(x_0)} (v(t,x)-v(t,y))_+w_+(t,y)\phi^2(y) |x - y|^{-d-2s}  \operatorname{d}\!x \operatorname{d}\!y \operatorname{d}\! t\\
  &\leq& 2\Lambda\int_{I_{r+\rho_1}^{\ominus}(t_0)} \int_{B_{r+\frac{\rho_2}{2}}(x_0)}w_+(t,y) \left(\int_{\mathbb{R}^d\setminus B_{r+\rho_2}(x_0)} w_+(t,x)|y - x|^{-d-2s}  \operatorname{d}\!x\right) \operatorname{d}\!y \operatorname{d}\! t\\
&\leq& C(\Lambda,d,s)(\frac{r+\rho_2}{\rho_2})^d\rho_2^{-2}\int_{I_{r+\rho_1}^{\ominus}(t_0)} \left(\int_{B_{r+\rho_2}(x_0)}w_+(t,x) \operatorname{d}\!x\right) \operatorname{Tail}(w_+(t);x_0,r+\rho_2) \operatorname{d}\! t\,.
\end{eqnarray*}
Substituting the estimates of $I$ and $II$ into \eqref{laste}, we derive
\begin{eqnarray}\label{est4}
   && \int_{t_0-(r+\rho_1)^2}^{\tau}\int_{\mathbb{R}^d} \int_{\mathbb{R}^d} (v(t,x)-v(t,y))\left(w_+(t,x)\phi^2(x)-w_+(t,y)\phi^2(y)\right)\eta^2(t) K(t,x,y) \operatorname{d}\!x \operatorname{d}\!y \operatorname{d}\! t\nonumber\\
    &\geq& C(\Lambda,d,s)r^{-d-2s} \int_{t_0-r^2}^{\tau} \int_{B_r(x_0)} \int_{B_r(x_0)} w_{+}(t,x) w_{-}(t,y) \operatorname{d}\! x \operatorname{d}\! y \operatorname{d}\! t\nonumber\\
     &&+\Lambda^{-1}\int_{t_0-r^2}^{\tau} \int_{B_{r}(x_0)} \int_{B_{r}(x_0)} \frac{\left|w_+(t,x)-w_+(t,y)\right|^2}{|x - y|^{d+2s}} \operatorname{d}\!x \operatorname{d}\!y\operatorname{d}\! t\nonumber\\
    && -C(\Lambda,d,s)\rho_2^{-2}\int_{I_{r+\rho_1}^{\ominus}(t_0)} \int_{B_{r+\rho_2}(x_0)} w^2_+(t,x)\operatorname{d}\!x \operatorname{d}\! t  \nonumber\\
   &&-C(\Lambda,d,s)(\frac{r+\rho_2}{\rho_2})^d\rho_2^{-2}\int_{I_{r+\rho_1}^{\ominus}(t_0)} \left(\int_{B_{r+\rho_2}(x_0)}w_+(t,x) \operatorname{d}\!x\right) \operatorname{Tail}(w_+(t);x_0,r+\rho_2) \operatorname{d}\! t\,.
\end{eqnarray}
Finally, combining \eqref{weaksol1}-\eqref{est3} with \eqref{est4}, we deduce that
\begin{eqnarray*}
&& \int_{B_r(x_0)} w^2_{+}(\tau,x) \operatorname{d}\! x +\int_{t_0-r^2}^{\tau}\int_{B_r(x_0)} |Dw_{+}(t,x)|^2\operatorname{d}\! x\operatorname{d}\! t \nonumber\\
 &&+\int_{t_0-r^2}^{\tau}\int_{B_r(x_0)} \int_{B_r(x_0)}\frac{|w_{+}(t,x)-w_{+}(t,y)|^2}{|x-y|^{d+2s}}\operatorname{d}\! x \operatorname{d}\! y\operatorname{d}\!t \nonumber\\
 &&+ r^{-d-2s} \int_{t_0-r^2}^{\tau} \int_{B_r(x_0)} \int_{B_r(x_0)} w_{+}(t,x) w_{-}(t,y) \operatorname{d}\! x \operatorname{d}\! y \operatorname{d}\! t\nonumber\\
 &\leq&  C \left(\rho_1^{-2}\vee\rho_2^{-2}\right)\int_{I_{r + \rho_1}^{\ominus}(t_0)} \int_{B_{r+\rho_2}(x_0)} w^2_{+}(t,x) \operatorname{d}\! x \operatorname{d}\! t \nonumber\\
&&+ C\left(\frac{r+\rho_2}{\rho_2}\right)^d \rho_2^{-2} \int_{I_{r+\rho_1}^{\ominus}(t_0)} \left(\int_{B_{r+\rho_2}(x_0)} w_{+}(t,x) \operatorname{d}\!x \right) \operatorname{Tail}(w_{+}(t);x_0,r+\rho_2)\operatorname{d}\!t
\end{eqnarray*}
holds for any $\tau\in[t_0-r^2,t_0]$, where the positive constant $C$ depends only on $d,\,s$ and $\Lambda$. Hence, we can conclude that the Caccioppoli inequality \eqref{eq:Cacc} is valid. The proof of \eqref{eq:Cacc2} is treated in  exactly the same way as above, the only difference being that $\eta(t)\equiv1$ is chosen at this point. Thus, we complete the proof of Lemma \ref{lemma:Cacc}\,.
\end{proof}

\section{H\"{o}lder estimates for homogeneous mixed problems}\label{section3}

This section is devoted to the proof of Theorem \ref{prop:Holderq} on the local H\"{o}lder estimate with an optimal tail term for weak solutions of the homogeneous mixed local and nonlocal parabolic equation \eqref{eq:PDE-hom}.

\subsection{Local boundedness with $L^1$-Tail}

In contrast to the existing local boundedness results for mixed parabolic problems, which rely on the $L^\infty$-Tail, the aim of this subsection is to establish an improved local boundedness estimate with $L^1$-Tail for the weak solution of \eqref{eq:PDE-hom}. For this purpose, we first derive the following lemma, which as a basic element modifies the $L^\infty$-Tail in the local boundedness estimate to the $L^1$-Tail.

\begin{lemma}\label{lemma:imptail}
Let $v$ be a local weak subsolution (supersolution) to \eqref{eq:PDE-hom} in $\Omega_T$,
then for any $R\in(0,1]$ and $(t_0,x_0)\in I\times\Omega$ with $I^\ominus_{2R}(t_0)\times B_{2R}(x_0)\Subset I\times\Omega$, there exists a positive constant $C$ depending only on $d,\,s,\,\Lambda$ such that
\begin{eqnarray}\label{imptail}
 \sup_{t\in I_{\frac{5R}{6}}^{\ominus}(t_0)}\left(\fint_{B_{\frac{5R}{6}}(x_0)} v^2_{\pm}(t,x) \operatorname{d}\! x \right)^{\frac{1}{2}}&\leq& C \left(\fint_{I_{R}^{\ominus}(t_0)} \fint_{B_{R}(x_0)} v^2_{\pm}(t,x) \operatorname{d}\! x \operatorname{d}\! t\right)^\frac{1}{2} \nonumber\\
&&+ C \fint_{I_{R}^{\ominus}(t_0)}
\operatorname{Tail}(v_{\pm}(t);x_0,R)\operatorname{d}\!t\,.
\end{eqnarray}
\end{lemma}

\begin{proof}
Here we only prove the desired estimate for the subsolution, which can be demonstrated in a similar way for the supersolution.
Let $\frac{5R}{6}\leq r<\rho\leq R$, $\rho_1=\frac{\rho-r}{2}$ and $\rho_2=\rho-r$, then by virtue of the Caccioppoli inequality \eqref{eq:Cacc}, we obtain
\begin{eqnarray*}
\sup_{t\in I_{r}^{\ominus}(t_0)}\fint_{B_r(x_0)} v^2_{+}(t,x) \operatorname{d}\! x&\leq& C(\rho-r)^{-2}\int_{I_{R}^{\ominus}(t_0)}\fint_{B_R(x_0)}v^2_{+}(t,x) \operatorname{d}\! x \operatorname{d}\! t\\
&&+C(\frac{\rho-r}{R})^{-d-2}\int_{I_{\rho}^{\ominus}(t_0)}
\left(\fint_{B_\rho(x_0)}v_{+}(t,x) \operatorname{d}\! x\right)\int_{\R^d\setminus B_\rho(x_0)}\frac{v_{+}(t,y)}{|y-x_0|^{d+2s}} \operatorname{d}\! y \operatorname{d}\! t,
\end{eqnarray*}
where the positive constant $C=C(d,s,\Lambda)$.
Applying Young's inequality and H\"{o}lder's inequality, the second term on the right side of the above estimate can be controlled by
\begin{eqnarray*}
   && C(\frac{\rho-r}{R})^{-d-2}\sup_{t\in I_{\rho}^{\ominus}(t_0)}\left(\fint_{B_\rho(x_0)}v_{+}(t,x) \operatorname{d}\! x\right) \int_{I_{\rho}^{\ominus}(t_0)}
\int_{\R^d\setminus B_\rho(x_0)}\frac{v_{+}(t,x)}{|x-x_0|^{d+2s}} \operatorname{d}\! x \operatorname{d}\! t \\
   &\leq& \frac{1}{2}\sup_{t\in I_{\rho}^{\ominus}(t_0)}\left(\fint_{B_\rho(x_0)}v_{+}(t,x) \operatorname{d}\! x\right)^2+C\left[(\frac{\rho-r}{R})^{-d-2}\int_{I_{\rho}^{\ominus}(t_0)}
\int_{\R^d\setminus B_\rho(x_0)}\frac{v_{+}(t,x)}{|x-x_0|^{d+2s}} \operatorname{d}\! x \operatorname{d}\! t\right]^2  \\
   &\leq&  \frac{1}{2}\sup_{t\in I_{\rho}^{\ominus}(t_0)}\fint_{B_\rho(x_0)}v^2_{+}(t,x) \operatorname{d}\! x+C(\frac{\rho-r}{R})^{-2(d+2)}\left(\int_{I_{R}^{\ominus}(t_0)}
\int_{\R^d\setminus B_{\frac{5R}{6}}(x_0)}\frac{v_{+}(t,x)}{|x-x_0|^{d+2s}} \operatorname{d}\! x \operatorname{d}\! t\right)^2.
\end{eqnarray*}
We denote
\begin{eqnarray*}
  c_1 &:=& C \int_{I_{R}^{\ominus}(t_0)}\fint_{B_R(x_0)}v^2_{+}(t,x) \operatorname{d}\! x \operatorname{d}\! t, \\
  c_2 &:=& CR^{2(d+2)}\left(\int_{I_{R}^{\ominus}(t_0)}
\int_{\R^d\setminus B_{\frac{5R}{6}}(x_0)}\frac{v_{+}(t,x)}{|x-x_0|^{d+2s}} \operatorname{d}\! x \operatorname{d}\! t\right)^2, \\
  f(r) &:=&  \sup_{t\in I_{r}^{\ominus}(t_0)}\fint_{B_r(x_0)} v^2_{+}(t,x) \operatorname{d}\! x,
\end{eqnarray*}
then it follows that
\begin{equation*}
  f(r)\leq \frac{1}{2}f(\rho)+\frac{c_1}{(\rho-r)^{2}}+\frac{c_2}{(\rho-r)^{2(d+2)}}
\end{equation*}
for any  $\frac{5R}{6}\leq r<\rho\leq R$. Note that the boundedness of $f$ is guaranteed by $v\in L_{\rm loc}^\infty(I;L^2_{\rm loc}(\Omega))$.
Then a classical iteration lemma, given by Lemma \ref{iterate3}\,, leads to
\begin{equation*}
  f(\frac{5R}{6})\leq C(d)\left(\frac{c_1}{R^{2}}+\frac{c_2}{R^{2(d+2)}}\right).
\end{equation*}
Finally, combining with the division of the integral domain and H\"{o}lder's inequality, we deduce that
\begin{eqnarray*}
&&\sup_{t\in I_{\frac{5R}{6}}^{\ominus}(t_0)}\left(\fint_{B_{\frac{5R}{6}}(x_0)} v^2_{+}(t,x) \operatorname{d}\! x \right)^{\frac{1}{2}}\\
&\leq& C \left(\fint_{I_{R}^{\ominus}(t_0)} \fint_{B_{R}(x_0)} v^2_{+}(t,x) \operatorname{d}\! x \operatorname{d}\! t\right)^\frac{1}{2} +C\int_{I_{R}^{\ominus}(t_0)}
\int_{\R^d\setminus B_{\frac{5R}{6}}(x_0)}\frac{v_{+}(t,x)}{|x-x_0|^{d+2s}} \operatorname{d}\! x \operatorname{d}\! t\\
&\leq&C \left(\fint_{I_{R}^{\ominus}(t_0)} \fint_{B_{R}(x_0)} v^2_{+}(t,x) \operatorname{d}\! x \operatorname{d}\! t\right)^\frac{1}{2} + C \fint_{I_{R}^{\ominus}(t_0)}
\operatorname{Tail}(v_{+}(t);x_0,R)\operatorname{d}\!t\,.
\end{eqnarray*}
Therefore, we complete the proof of Lemma \ref{lemma:imptail}\,.
\end{proof}

Now we demonstrate the local boundedness estimate with $L^1$-Tail for the weak solution of the mixed parabolic equation \eqref{eq:PDE-hom}.
\begin{lemma}
\label{lemma:locbd}
Let $v$ be a local weak solution to \eqref{eq:PDE-hom} in $\Omega_T$,
then for any $R\in(0,1]$ and $(t_0,x_0)\in I\times\Omega$ with $I^\ominus_{2R}(t_0)\times B_{2R}(x_0)\Subset I\times\Omega$, there exists a positive constant $C$, depending only on $d,\,s,\,\Lambda$, such that
\begin{eqnarray}\label{locbd}
  \sup_{Q_{\frac{R}{2}}(t_0,x_0)} |v(x,t)| \leq C\left( \fint_{Q_R(t_0,x_0)} |v(t,x)| \operatorname{d}\!x\operatorname{d}\!t +  \fint_{I_R^{\ominus}(t_0)} \operatorname{Tail}(v(t);x_0,R)\operatorname{d}\!t \right).
\end{eqnarray}
\end{lemma}

\begin{proof} In order to prove \eqref{locbd},
we first claim that
\begin{equation}\label{bdd}
  \sup_{Q_{\frac{R}{2}}(t_0,x_0)} |v| \leq
  C \delta^{-\frac{d+2}{4}}\left(\fint_{Q_R(t_0,x_0)} |v(t,x)|^2 \operatorname{d}\!x\operatorname{d}\!t\right)^{\frac{1}{2}} + C\fint_{ I_R^{\ominus}(t_0)} \operatorname{Tail}(v(t);x_0,R)\operatorname{d}\!t
\end{equation}
for any $\delta\in(0,1]$, where $C=C(d,s,\Lambda)$.

Let $r,\, \rho>0$ such that $\rho\leq\frac{R}{2}\leq r\leq r+\rho\leq R$.
Selecting $\rho_1=\rho_2=\rho$,
then by analogy to the proof of Lemma \ref{lemma:Cacc}\,, we can deduce that
\begin{eqnarray*}
&& \sup_{t\in I_{r}^{\ominus}(t_0)}\int_{B_r(x_0)} w^2_{+}(t,x) \operatorname{d}\! x +\int_{I_{r}^{\ominus}(t_0)}\int_{B_{r+\rho}(x_0)} |D(w_{+}\phi)(t,x)|^2\operatorname{d}\! x\operatorname{d}\! t \nonumber\\
 &\leq&  C \rho^{-2}\int_{I_{r + \rho}^{\ominus}(t_0)} \int_{B_{r+\rho}(x_0)} w^2_{+}(t,x) \operatorname{d}\! x \operatorname{d}\! t \nonumber\\
&&+ C\left(\frac{r+\rho}{\rho}\right)^d \rho^{-2} \int_{I_{r+\rho}^{\ominus}(t_0)} \int_{B_{r+\rho}(x_0)} w_{+}(t,x) \operatorname{d}\!x \operatorname{d}\!t \sup_{t\in I_{r+\rho}^{\ominus}(t_0)}\operatorname{Tail}(w_{+}(t);x_0,r+\rho),
\end{eqnarray*}
where the positive constant $C=C(d,\Lambda,s)$. To proceed, a combination of H\"{o}lder's inequality with Sobolev embedding theorem and Young's inequality yields that
\begin{eqnarray}\label{bdd2}
  && \|w^2_+\|_{L^{1+\frac{2}{d}}(I_{r}^{\ominus}(t_0)\times B_{r}(x_0))}\nonumber\\
   &\leq&C(d)\sup_{t\in I_{r}^{\ominus}(t_0)}\left(\int_{B_r(x_0)} w^2_{+}(t,x) \operatorname{d}\! x\right)^{\frac{2}{d+2}}\left(\int_{I_{r}^{\ominus}(t_0)}\int_{B_{r+\rho}(x_0)} |D(w_{+}\phi)(t,x)|^2\operatorname{d}\! x\operatorname{d}\! t\right)^{\frac{d}{d+2}}\nonumber\\
  &\leq&C \sup_{t\in I_{r}^{\ominus}(t_0)}\int_{B_r(x_0)} w^2_{+}(t,x) \operatorname{d}\! x +\int_{I_{r}^{\ominus}(t_0)}\int_{B_{r+\rho}(x_0)} |D(w_{+}\phi)(t,x)|^2\operatorname{d}\! x\operatorname{d}\! t\nonumber\\
 &\leq&  C \sigma(r,\rho)\left(\|w^2_+\|_{L^1(I_{r+\rho}^{\ominus}(t_0)\times B_{r+\rho}(x_0))}\right.\nonumber\\
 &&\left.+ \|w_+\|_{L^1(I_{r+\rho}^{\ominus}(t_0)\times B_{r+\rho}(x_0))}\sup_{t\in I_{r+\rho}^{\ominus}(t_0)}\operatorname{Tail}(w_{+}(t);x_0,r+\rho)\right)\,,
\end{eqnarray}
where $\sigma(r,\rho):=\left(\frac{r+\rho}{\rho}\right)^d\rho^{-2}$ and the positive constant $C=C(d,\Lambda,s)>0$. For $l>0$, we denote $w_l(t,x):=(v(t,x)-l)_+$ and set
$$|A(l,r)|:= \int_{I_{r}^{\ominus}(t_0)} \left|\{x\in B_r(x_0) \mid v(t,x)>l\}\right|\operatorname{d}\! t\,,$$
then it follows from H\"{o}lder's inequality and \eqref{bdd2} that
\begin{eqnarray}\label{bdd3}
\|w^2_l\|_{L^1(I_{r}^{\ominus}(t_0)\times B_{r}(x_0))} &\leq& |A(l,r)|^{\frac{2}{d+2}}
\|w^2_l\|_{L^{1+\frac{2}{d}}(I_{r}^{\ominus}(t_0)\times B_{r}(x_0))}\nonumber\\
&\leq&C  |A(l,r)|^{\frac{2}{d+2}}\sigma(r,\rho)\left(\|w^2_l\|_{L^1(I_{r+\rho}^{\ominus}(t_0)\times B_{r+\rho}(x_0))}\right.\nonumber\\
&&\left. + \|w_l\|_{L^1(I_{r+\rho}^{\ominus}(t_0)\times B_{r+\rho}(x_0))}
  \sup_{t\in I_{r+\rho}^{\ominus}(t_0)}\operatorname{Tail}(v_+(t);x_0,r+\rho)\right)\,.
\end{eqnarray}
Let $k\in(0,l)$ be arbitrary, then by means of Chebyshev's inequality and a direct calculation one can derive the following
\begin{eqnarray*}
  \|w^2_l\|_{L^1(I_{r+\rho}^{\ominus}(t_0)\times B_{r+\rho}(x_0))} &\leq& \|w^2_k\|_{L^1(I_{r+\rho}^{\ominus}(t_0)\times B_{r+\rho}(x_0))}\,, \\
  \|w_l\|_{L^1(I_{r+\rho}^{\ominus}(t_0)\times B_{r+\rho}(x_0))} &\leq& \frac{\|w^2_k\|_{L^1(I_{r+\rho}^{\ominus}(t_0)\times B_{r+\rho}(x_0))}}{l-k}\,, \\
   |A(l,r)|&\leq& \frac{\|w^2_k\|_{L^1(I_{r+\rho}^{\ominus}(t_0)\times B_{r+\rho}(x_0))}}{(l-k)^2}\,.
\end{eqnarray*}
Substituting the above three inequalities into \eqref{bdd3}, we obtain
\begin{eqnarray}\label{bdd4}
\|w^2_l\|_{L^1(I_{r}^{\ominus}(t_0)\times B_{r}(x_0))}
&\leq&C(l-k)^{-\frac{4}{d+2}}\sigma(r,\rho) \|w^2_k\|_{L^1(I_{r+\rho}^{\ominus}(t_0)\times B_{r+\rho}(x_0))}^{1+\frac{2}{d+2}}\nonumber\\
&&\times \left(1+\frac{\displaystyle\sup_{t\in I_{r+\rho}^{\ominus}(t_0)}\operatorname{Tail}(v_+(t);x_0,r+\rho)}{l-k}\right),
\end{eqnarray}
where the positive constant $C$ depends only on $d$, $s$, and $\Lambda$.

The next step is to build a De Giorgi iteration scheme to iterate the key estimate \eqref{bdd4}. For any $i\in \mathbb{N}$, we define
\begin{eqnarray*}
   && l_i=M(1-2^{-i}),\, \, w_i=(v-l_i)_+, \,\, A_i=\|w^2_i\|_{L^1(I_{r_i}^{\ominus}(t_0)\times B_{r_i}(x_0))},\\
   && \rho _i=2^{-i-1}R, \,\, r_0=\frac{3R}{4}, \,\, r_{i+1}=r_i-\rho_{i+1}=\frac{R}{2}\left(1+\big(\frac{1}{2}\big)^{i+2}\right),
\end{eqnarray*}
where $M>0$ is to be determined later. It is easy to verify that $\rho_i\leq\frac{R}{2}\leq r_i\leq r_i+\rho_i\leq \frac{3R}{4}$, and $\rho_i\searrow 0$, $r_i\searrow \frac{R}{2}$, $l_i\nearrow M$ as $i\rightarrow\infty$.
Hence, by virtue of \eqref{bdd4}, we deduce that
\begin{eqnarray*}
  A_{i+1} &\leq&\frac{C\sigma(r_{i+1},\rho_{i+1})}{(l_{i+1}-l_{i})^{\frac{4}{d+2}}} A_{i}^{1+\frac{2}{d+2}}\left(1+
  \frac{\displaystyle\sup_{t\in I_{r_{i+1}+\rho_{i+1}}^{\ominus}(t_0)}\operatorname{Tail}(v_+(t);x_0,r_{i+1}+\rho_{i+1})}{l_{i+1}-l_{i}}\right) \\
 &\leq&\frac{C2^{(d+3+\frac{4}{d+2})i}}{R^{2}M^{\frac{4}{d+2}}} A_{i}^{1+\frac{2}{d+2}}\left(1+
  \frac{\displaystyle\sup_{t\in I_{\frac{3R}{4}}^{\ominus}(t_0)}\operatorname{Tail}(v_+(t);x_0,\frac{R}{2})}{M}\right) \,.
\end{eqnarray*}
For any fixed $\delta\in(0,1]$, we first select
  $$ M\geq \delta\displaystyle\sup_{t\in I_{\frac{3R}{4}}^{\ominus}(t_0)}\operatorname{Tail}(v_+(t);x_0,\frac{R}{2}).$$
Denoting $b=2^{d+3+\frac{4}{d+2}}>1$, then we calculate that
\begin{equation}\label{bdd5}
  A_{i+1} \leq\frac{C_1 }{\delta R^{2}M^{\frac{4}{d+2}}} b^i A_{i}^{1+\frac{2}{d+2}},
\end{equation}
where $C_1$ is a positive constant depending exclusively on $d$, $s$, and $\Lambda$. We further choose
\begin{equation*}
  M=\delta\displaystyle\sup_{t\in I_{\frac{3R}{4}}^{\ominus}(t_0)}\operatorname{Tail}(v_+(t);x_0,\frac{R}{2})
  +C_1^{\frac{d+2}{4}}b^{\frac{(d+2)^2}{8}}\delta^{-\frac{d+2}{4}}R^{-\frac{d+2}{2}}A_0^{\frac{1}{2}},
\end{equation*}
then it leads to
\begin{equation}\label{bdd6}
  A_{0} \leq\left(\frac{C_1}{\delta R^{2}M^{\frac{4}{d+2}}}\right)^{-\frac{d+2}{2}}
   b^{-\frac{(d+2)^2}{4}}\,.
\end{equation}
In terms of \eqref{bdd5} and \eqref{bdd6}, we utilize the iteration lemma given by Lemma \ref{iterate1} to derive $\displaystyle\lim_{i\rightarrow\infty}A_i=0$, which implies that
\begin{eqnarray}\label{bdd7}
  \sup_{I_{\frac{R}{2}}^{\ominus}\times B_{\frac{R}{2}}}v_+(t,x) \leq C\delta^{-\frac{d+2}{4}}\left(\fint_{I_{\frac{3R}{4}}^{\ominus}(t_0)} \fint_{B_{\frac{3R}{4}}(x_0)} v_+^2(t,x) \operatorname{d}\! x \operatorname{d}\! t\right)^{\frac{1}{2}}+
   \delta\displaystyle\sup_{t\in I_{\frac{3R}{4}}^{\ominus}(t_0)}\operatorname{Tail}(v_+(t);x_0,\frac{R}{2}),
\end{eqnarray}
where the positive constant $C=C(d,\Lambda,s)$. The above proof process shows that \eqref{bdd7} also holds for the local weak subsolution of \eqref{eq:PDE-hom}.

The next objective is to improve the $L^\infty$-Tail in \eqref{bdd7} to the $L^1$-Tail.
Let $\psi\in C_0^\infty(B_{\frac{5R}{6}}(x_0))$ be a smooth cutoff function with $\psi\equiv1$ in $B_{\frac{4R}{5}}(x_0)$.
By a direct calculation, we obtain
\begin{equation*}
  \partial_t(\psi v)-\operatorname{div}(\mathcal{A}(t,x,D(\psi v)))+\mathcal{L}(\psi v)=-\mathcal{L}((1-\psi) v)=: f(t,x),\,\,\mbox{in}\,\, Q_{\frac{3R}{4}}(t_0,x_0).
\end{equation*}
It is then easy to verify that $$\underline{v}(t,x):=\psi(x)v(t,x)-\int^t_{t_0-(\frac{3R}{4})^2}\|f_+(\tau)\|_{L^\infty(B_{\frac{3R}{4}}(x_0))}\operatorname{d}\! \tau$$
is a subsolution of \eqref{eq:PDE-hom} in $Q_{\frac{3R}{4}}(t_0,x_0)$. Note that
\begin{equation*}
 (\psi(x)v(t,x))_+-\int^t_{t_0-(\frac{3R}{4})^2}\|f_+(\tau)\|_{L^\infty(B_{\frac{3R}{4}}(x_0))}\operatorname{d}\! \tau \leq(\underline{v}(t,x))_+\leq(\psi(x)v(t,x))_+,
\end{equation*}
it follows from \eqref{bdd7} and the choice of $\psi$ that
\begin{eqnarray}\label{bdd7-2}
  \sup_{I_{\frac{R}{2}}^{\ominus}(t_0)\times B_{\frac{R}{2}}(x_0)}v_+(t,x) &\leq &
  \int_{I_{\frac{3R}{4}}^\ominus(t_0)}\|f_+(t)\|_{L^\infty(B_{\frac{3R}{4}}(x_0))}\operatorname{d}\! t+
  C\delta^{-\frac{d+2}{4}}\left(\fint_{Q_{\frac{3R}{4}}(t_0,x_0)} v_+^2(t,x) \operatorname{d}\! x \operatorname{d}\! t\right)^{\frac{1}{2}}\nonumber\\
  &&+
   \delta\displaystyle\sup_{t\in I_{\frac{3R}{4}}^{\ominus}(t_0)}\operatorname{Tail}((\psi v)_+(t);x_0,\frac{R}{2}).
\end{eqnarray}
In order to estimate the first term on the right side of \eqref{bdd7-2}, we utilize the ellipticity condition \eqref{ellip} and the choice of $\psi$ to derive
\begin{eqnarray*}
f_+(t,x)&=&[-\mathcal{L}((1-\psi) v)]_+  \\
   &\leq&  C(\Lambda)\int_{\R^d\setminus B_{\frac{4R}{5}}(x_0)}\frac{v_+(t,y)}{|x-y|^{d+2s}}\operatorname{d}\! y\\
   &\leq&  C(\Lambda,d,s)\int_{\R^d\setminus B_{\frac{3R}{4}}(x_0)}\frac{v_+(t,y)}{|x_0-y|^{d+2s}}\operatorname{d}\! y
\end{eqnarray*}
for any $x\in B_{\frac{3R}{4}}(x_0)$. Combining the definition of the tail term with H\"{o}lder's inequality then gives
\begin{eqnarray}\label{bdd7-3}
 && \int_{I_{\frac{3R}{4}}^\ominus(t_0)}\|f_+(t)\|_{L^\infty(B_{\frac{3R}{4}}(x_0))}\operatorname{d}\! t\nonumber\\
 &\leq& C\fint_{I_{\frac{3R}{4}}^\ominus(t_0)}\operatorname{Tail}(v_+(t);x_0,\frac{3R}{4})\operatorname{d}\! t\nonumber\\
  &\leq&C\left( \fint_{Q_{R}(t_0,x_0)} v^2_{+}(t,x) \operatorname{d}\! x \operatorname{d}\! t\right)^\frac{1}{2}+C\fint_{I_{R}^\ominus(t_0)}\operatorname{Tail}(v_+(t);x_0,R)\operatorname{d}\! t\,.
\end{eqnarray}
We proceed to estimate the last term on the right side of \eqref{bdd7-2}. Applying the choice of $\psi$ and H\"{o}lder's inequality again, together with Lemma \ref{lemma:imptail}\,, we arrive at
\begin{eqnarray}\label{bdd7-4}
 && \displaystyle\sup_{t\in I_{\frac{3R}{4}}^{\ominus}(t_0)}\operatorname{Tail}((\psi v)_+(t);x_0,\frac{R}{2}) \nonumber \\
   &\leq& C(d)\displaystyle\sup_{t\in I_{\frac{5R}{6}}^{\ominus}(t_0)}\left(\fint_{B_{\frac{5R}{6}}^{\ominus}(x_0)} v_+^2(t,x) \operatorname{d}\! x \right)^{\frac{1}{2}} \nonumber  \\
  &\leq&C \left( \fint_{Q_{R}(t_0,x_0)} v^2_{+}(t,x) \operatorname{d}\! x \operatorname{d}\! t\right)^\frac{1}{2} + C \fint_{I_{R}^{\ominus}(t_0)}
\operatorname{Tail}(v_{+}(t);x_0,R)\operatorname{d}\!t\,,
\end{eqnarray}
where the positive constant $C=C(d,s,\Lambda)$.
Now substituting \eqref{bdd7-3} and \eqref{bdd7-4} into \eqref{bdd7-2}, we obtain
\begin{equation}\label{bdd7-5}
  \sup_{I_{\frac{R}{2}}^{\ominus}(t_0)\times B_{\frac{R}{2}}(x_0)}v_+(t,x) \leq
  C\delta^{-\frac{d+2}{4}}\left(\fint_{Q_{R}(t_0,x_0)} v_+^2(t,x) \operatorname{d}\! x \operatorname{d}\! t\right)^{\frac{1}{2}}+
   C\fint_{I_{R}^{\ominus}(t_0)}\operatorname{Tail}( v_+(t);x_0,R)\operatorname{d}\!t\,,
\end{equation}
where the constant $C=C(d,s,\Lambda)>0$.
The proof follows the same arguments as before, we can also show that
\begin{equation*}
  \sup_{I_{\frac{R}{2}}^{\ominus}(t_0)\times B_{\frac{R}{2}}(x_0)}v_-(t,x) \leq
  C\delta^{-\frac{d+2}{4}}\left(\fint_{Q_{R}(t_0,x_0)} v_-^2(t,x) \operatorname{d}\! x \operatorname{d}\! t\right)^{\frac{1}{2}}+
   C\fint_{I_{R}^{\ominus}(t_0)}\operatorname{Tail}( v_-(t);x_0,R)\operatorname{d}\!t\,.
\end{equation*}
Therefore, we confirm that the assertion \eqref{bdd} is valid, which implies that the weak solution $v$ is locally bounded.

Finally, we aim to reduce the exponent in the first term on the right side of \eqref{bdd}. A combination of \eqref{bdd} and Young's inequality yields that
\begin{equation}\label{bdd8}
  \sup_{Q_{\frac{R}{2}}(t_0,x_0)} |v| \leq \delta \sup_{Q_{R}(t_0,x_0)} |v|+
  C(d,s,\Lambda,\delta)\fint_{Q_{R}(t_0,x_0)} |v(t,x)| \operatorname{d}\! x \operatorname{d}\! t+
   C(d,s,\Lambda)\fint_{I_{R}^{\ominus}(t_0)}\operatorname{Tail}( v(t);x_0,R)\operatorname{d}\!t\,.
\end{equation}
Let $\frac{R}{2}\leq\varrho_1<\varrho_2\leq R$, performing a standard covering argument on \eqref{bdd8}, we can deduce that
\begin{eqnarray*}
  \sup_{Q_{\varrho_1}(t_0,x_0)} |v| \leq C_2\delta \sup_{Q_{\varrho_2}(t_0,x_0)} |v|+
  \frac{C}{(\varrho_2-\varrho_1)^{d+2}}\left(\int_{Q_{R}(t_0,x_0)} |v(t,x)| \operatorname{d}\! x \operatorname{d}\! t+R^d\int_{I_{R}^{\ominus}(t_0)}\operatorname{Tail}( v(t);x_0,R)\operatorname{d}\!t\right)\,,
\end{eqnarray*}
where the positive constants $C_2=C_2(d,s,\Lambda)$ and $C=C(d,s,\Lambda,\delta)$.
If we choose $\delta=\frac{1}{2C_2}$, then the first term on the right side can be absorbed by using the iteration Lemma \ref{iterate3}\,, which implies that the local boundedness estimate \eqref{locbd} with $L^1$-Tail holds. This completes the proof of Lemma \ref{lemma:locbd}\,.
\end{proof}

\subsection{Local H\"older regularity with an optimal Tail}

With the Caccioppoli inequality and local boundedness at our disposal,
it is crucial for us to further establish a growth lemma to prove H\"older continuity. Now we start with a so-called De Giorgi type of lemma, which states that if a measure density is achieved in a parabolic cylinder, then pointwise information can be extracted in a smaller parabolic cylinder with the same vertex.

\begin{lemma}
\label{lemma:growth-lemma1}
Let $v$ be a local weak supersolution to \eqref{eq:PDE-hom} in $\Omega_T$.
For every $R\in(0,1]$ and $(t_0,x_0)\in I\times\Omega$ with $I^\ominus_{2R}(t_0)\times B_{2R}(x_0)\Subset I\times\Omega$, any $\delta \in (0,1]$, $q\in(1,2]$ and $H > 0$, there exists $\lambda\in (0,1)$ depending only on $d,\,s,\,\Lambda,\,q,\,\delta$, such that if  $v \geq 0$ in $Q_{2R}(t_0,x_0)$, as well as
\begin{equation}\label{ularge}
\left| \left\{ v \leq H \right\} \cap I^{\ominus}_{\delta R}(t_0) \times B_R (x_0) \right| \leq \lambda |I^{\ominus}_{\delta R} (t_0)\times B_R(x_0)|
\end{equation}
and
\begin{equation}\label{eq:Tail-small}
\left(\fint_{I_{2R}^{\ominus}(t_0)} \operatorname{Tail}(v_-(t);x_0,2R)^q \operatorname{d}\! t \right)^{\frac{1}{q}} \leq H,
\end{equation}
then it holds that
\begin{equation*}
v(t,x) \geq \frac{H}{2} ~~ \text{ in } I^{\ominus}_{\frac{\delta R}{2}}(t_0) \times B_{\frac{R}{2}}(x_0).
\end{equation*}
\end{lemma}

\begin{proof}
For $i\in\mathbb{N_+}$, we define the following sequences
\begin{eqnarray*}
&& k_i = \frac{H}{2} + \frac{H}{2^i}, \,\, w_i = (v-k_{i})_-, \,\, R_i = \frac{R}{2} + \frac{R}{2^i},\\
&& A_i= \frac{\left| \left\{ v \leq k_i \right\} \cap I^{\ominus}_{\delta R_i} (t_0) \times B_{R_i} (x_0)\right|}{|I^{\ominus}_{\delta R_i} \times B_{R_i}|}, \\
&& B_i = \left( \fint_{I_{\delta R_i}^{\ominus}(t_0)} \left( \frac{|\{ v(t) \leq k_i \} \cap B_{R_i}(x_0)|}{|B_{R_i}|} \right)^{\frac{q}{q-1}}\operatorname{d}\! t \right)^{\frac{(q-1)d}{q(d+2)}}.
\end{eqnarray*}
Combining the Caccioppoli inequality \eqref{eq:Cacc} with H\"{o}lder's inequality and \eqref{eq:Tail-small}, we directly compute
\begin{eqnarray}\label{growth1-1}
&&\sup_{t\in I_{\delta R_{i+1}}^{\ominus}(t_0)}\int_{B_{R_{i+1}}(x_0)} w^2_{i}(t,x) \operatorname{d}\! x +\int_{I_{\delta R_{i+1}}^{\ominus}(t_0)}\int_{B_{ R_{i+1}}(x_0)} |Dw_{i}(t,x)|^2\operatorname{d}\! x\operatorname{d}\! t\nonumber\\
 &\leq&  C 2^{2i}(\delta R)^{-2} \int_{I_{\delta R_i}^{\ominus}(t_0)} \int_{B_{R_i}(x_0)} w^2_{i}(t,x) \operatorname{d}\! x \operatorname{d}\! t\nonumber\\
 &&+ C2^{i(d+2)}R^{-2}  \int_{I_{\delta R_i}^{\ominus}(t_0)} \left(\int_{B_{R_i}(x_0)} w_{i}(t,x) \operatorname{d}\!x \right) \operatorname{Tail}(w_{i}(t);x_0,R_i)\operatorname{d}\!t\nonumber\\
 &\leq&C2^{i(d+2)}R^{-2}H^2|Q_{R_i}|A_i+C2^{i(d+2)}R^{-2}H\int_{I_{\delta R_i}^{\ominus}(t_0)} \left(\int_{B_{R_i}(x_0)} \chi_{\{v\leq k_i\}}(t,x) \operatorname{d}\!x\right)\operatorname{Tail}(w_{i}(t);x_0,2R)\operatorname{d}\!t
 \nonumber\\
 &\leq&C2^{i(d+2)}R^{-2}H^2|Q_{R_i}|A_i+C2^{i(d+2)}R^{-2}H\int_{I_{\delta R_i}^{\ominus}(t_0)} \left(\int_{B_{R_i}(x_0)} \chi_{\{v\leq k_i\}}(t,x) \operatorname{d}\!x\right)\operatorname{Tail}(v_{-}(t);x_0,2R)\operatorname{d}\!t
 \nonumber\\
 &\leq&C2^{i(d+2)}R^{-2}H^2|Q_{R_i}|\left(A_i+B_i^{1+\frac{2}{d}}\right)\,.
\end{eqnarray}
Applying H\"{o}lder's inequality repeatedly, in combination with the Sobolev embedding theorem, Young's inequality and \eqref{growth1-1}, we deduce that
\begin{eqnarray*}
 && A_{i+1}|I^{\ominus}_{\delta R_{i+1}} \times B_{R_{i+1}}|\\
  &\leq& \int_{I_{\delta R_{i+1}}^{\ominus}(t_0)} \int_{B_{R_{i+1}}(x_0)}\frac{w_i^2}{(k_i-k_{i+1})^2}\operatorname{d}\!x\operatorname{d}\!t  \\
   &\leq& C2^{2i}H^{-2} \left( A_{i}|I^{\ominus}_{\delta R_{i}} \times B_{R_{i}}|\right)^{\frac{2}{d+2}} \left[\int_{I_{\delta R_{i+1}}^{\ominus}(t_0)} \left(\int_{B_{R_{i+1}}(x_0)}w_i^{2}\operatorname{d}\!x\right)^{\frac{2}{d}}
   \left(\int_{B_{R_{i+1}}(x_0)}w_i^{\frac{2d}{d-2}}\operatorname{d}\!x\right)^{\frac{d-2}{d}}\operatorname{d}\!t\right]^{\frac{d}{d+2}} \\
   &\leq& C2^{2i}H^{-2} \left( A_{i}|I^{\ominus}_{\delta R_{i}} \times B_{R_{i}}|\right)^{\frac{2}{d+2}} \\
   &&\times\left[\int_{I_{\delta R_{i+1}}^{\ominus}(t_0)} \left(\int_{B_{R_{i+1}}(x_0)}w_i^{2}\operatorname{d}\!x\right)^{\frac{2}{d}}
   \left(R_{i+1}^{-2}\int_{B_{R_{i+1}}(x_0)}w_i^{2}\operatorname{d}\!x
   +\int_{B_{R_{i+1}}(x_0)}|Dw_i|^{2}\operatorname{d}\!x\right)\operatorname{d}\!t\right]^{\frac{d}{d+2}}\\
    &\leq& C2^{2i}H^{-2} \left( A_{i}|I^{\ominus}_{\delta R_{i}} \times B_{R_{i}}|\right)^{\frac{2}{d+2}}\left(\sup_{t\in I_{\delta R_{i+1}}^{\ominus}(t_0)}\int_{B_{R_{i+1}}(x_0)} w^2_{i} \operatorname{d}\! x +\int_{I_{\delta R_{i+1}}^{\ominus}(t_0)}\int_{B_{ R_{i+1}}(x_0)} |Dw_{i}|^2\operatorname{d}\! x\operatorname{d}\! t\right)\\
   &\leq&C2^{(d+4)i}\delta^{\frac{4}{d+2}}|Q_{R_i}|\left(A_i^{1+\frac{2}{d+2}}+A_i^{\frac{2}{d+2}}B_i^{1+\frac{2}{d}}\right).
\end{eqnarray*}
Then it follows that
\begin{equation}\label{growth1-2}
  A_{i+1}\leq \frac{C_3}{\delta^2}2^{(d+4)i}\left(A_i^{1+\frac{2}{d+2}}+A_i^{\frac{2}{d+2}}B_i^{1+\frac{2}{d}}\right),
\end{equation}
where the positive constant $C_3$ only depends on $d$, $s$ and $\Lambda$.
By analogy with the estimates of $A_{i+1}$ in \eqref{growth1-2}, we can further obtain
\begin{equation}\label{growth1-3}
  B_{i+1}\leq \frac{C_4}{\delta^2}2^{(d+4)i}\left(A_i+B_i^{1+\frac{2}{d}}\right),
\end{equation}
where the constant $C_4=C_4(d,s,\Lambda)>0$. Note that
\begin{equation*}
  A_1+B_1^{1+\frac{2}{d}}\leq  A_1+A_1^{\frac{q-1}{q}}\leq 2A_1^{\frac{q-1}{q}}\leq  2 \lambda^{\frac{q-1}{q}},
\end{equation*}
which is ensured by H\"{o}lder's inequality and the assumption \eqref{ularge}.
From this, combining \eqref{growth1-2} with \eqref{growth1-3}, and applying a classical iteration Lemma \ref{iterate2} to derive that there exists $\lambda >0$ sufficiently small, depending only on $d,\,s,\,\Lambda,\,\delta,\,q$, such that
the sequences $\{A_i\}$ and $\{B_i\}$ tend to zero as $i\rightarrow \infty$.
Hence, we conclude that
\begin{equation*}
v(t,x) \geq \frac{H}{2} ~~ \text{ in } I^{\ominus}_{\frac{\delta R}{2}}(t_0) \times B_{\frac{R}{2}}(x_0),
\end{equation*}
which completes the proof of Lemma \ref{lemma:growth-lemma1}\,.
\end{proof}

In the next step we establish a lemma that shows the forward propagation of the measure theoretical information in time.
\begin{lemma}
\label{lemma:growth-lemma2}
Let $\alpha \in (0,1]$ and
$v$ be a local weak supersolution to \eqref{eq:PDE-hom} in $\Omega_T$.
For every $R\in(0,1]$ and $(t_0,x_0)\in I\times\Omega$ with $I^\ominus_{2R}(t_0)\times B_{2R}(x_0)\Subset I\times\Omega$, any $q\in(1,2]$  and $H > 0$, if $v \ge 0$ in $Q_{2R}(t_0,x_0)$
and
\begin{equation}
\label{eq:measure-ass1}
\left| \left\{ v(t_1,\cdot) \ge H \right\} \cap B_R (x_0)\right| \ge \alpha |B_R(x_0)| \,\,\mbox{for some}\,\, t_1 \in I_{2 R}^{\ominus}(t_0) ,
\end{equation}
then either
\begin{equation}
\label{Tail-large}
\left(\fint_{I_{2R}^{\ominus}(t_0)} \operatorname{Tail}(v_-(t);x_0,2R)^q \operatorname{d}\! t \right)^{\frac{1}{q}}> H,
\end{equation}
or there exists $\varepsilon,\,\delta\in(0,1)$ depending only on $d,\,s,\, \Lambda,\,q,\,\alpha$, such that
\begin{equation*}
\left| \left\{ v (t,\cdot)\geq \varepsilon H \right\} \cap B_R(x_0) \right| \geq \frac{\alpha}{2}| B_R(x_0)|
\end{equation*}
for almost everywhere $t\in I_{\delta R}^{\oplus}(t_1)\subset I_{2R}^{\ominus}(t_0)$.
\end{lemma}

\begin{proof}
We may assume that \eqref{Tail-large} is invalid, then a combination of the Caccioppoli inequality \eqref{eq:Cacc2}, H\"{o}lder's inequality and the assumption \eqref{eq:measure-ass1} yields that
\begin{eqnarray*}
 && \int_{B_{(1-\epsilon)R}(x_0)} (v-H)^2_{-}(t,x) \operatorname{d}\! x \nonumber\\
 &\leq&  \int_{B_{R}(x_0)} (v-H)^2_{-}(t_1,x) \operatorname{d}\! x+ C (\epsilon R)^{-2}\int_{I_{\delta R}^{\oplus}(t_1)} \int_{B_{R}(x_0)} (v-H)^2_{-}(t,x) \operatorname{d}\! x \operatorname{d}\! t \nonumber\\
&&+ C\epsilon^{-d-2} R^{-2} \int_{I_{\delta R}^{\oplus}(t_1)} \left(\int_{B_{R}(x_0)} (v-H)_{-}(t,x) \operatorname{d}\!x \right) \operatorname{Tail}((v-H)_{-}(t);x_0,R)\operatorname{d}\!t\nonumber\\
&\leq&\left((1-\alpha)+C\epsilon^{-2}\delta^2+C\epsilon^{-d-2}\delta^{\frac{2(q-1)}{q}}\right)H^2|B_R|
\end{eqnarray*}
holds for almost everywhere  $t\in I_{\delta R}^{\oplus}(t_1)$, where the constant $\epsilon \in(0,1)$ is to be determined later.
From this, a direct calculation leads to
\begin{eqnarray*}
   && \left| \left\{ v (t,\cdot)\leq \varepsilon H \right\} \cap B_R(x_0) \right| \\
   &\leq& \left| \left\{ v (t,\cdot)\leq \varepsilon H \right\} \cap B_{(1-\epsilon)R}(x_0) \right|+ \left| B_{R}(x_0)\setminus B_{(1-\epsilon)R}(x_0)\right| \\
   &\leq& ((1-\varepsilon)H)^{-2} \int_{B_{(1-\epsilon)R}(x_0)} (v-H)^2_{-}(t,x) \operatorname{d}\! x+d\epsilon |B_R|\\
   &\leq&(1-\varepsilon)^{-2}\left((1-\alpha)+C\epsilon^{-2}\delta^2+C\epsilon^{-d-2}\delta^{\frac{2(q-1)}{q}}+d\epsilon\right) |B_R|.
\end{eqnarray*}
To proceed, we choose appropriate $\epsilon$, $\delta$ and $\varepsilon$, depending only on $d,\,s,\, \Lambda,\,q,\,\alpha$, such that
\begin{equation*}
  \left| \left\{ v (t,\cdot)\leq \varepsilon H \right\} \cap B_R(x_0) \right| \leq (1-\frac{\alpha}{2})|B_R|
\end{equation*}
for almost all  $t\in I_{\delta R}^{\oplus}(t_1)$.
This implies that the desired result is true, thus completing the proof of Lemma \ref{lemma:growth-lemma2}\,.
\end{proof}
In the sequel, we establish a measure shrinking lemma.
\begin{lemma}
\label{lemma:growth-lemma3}
Let $\alpha \in (0,1]$ and let
$v$ be a local weak supersolution to \eqref{eq:PDE-hom} in $\Omega_T$.
For every $R>0$ and $(t_0,x_0)\in I\times\Omega$ with $I^\ominus_{2R}(t_0)\times B_{2R}(x_0)\Subset I\times\Omega$, and any $\delta,\, \sigma \in (0,1]$, $q\in(1,2]$, and $H > 0$, if $v \ge 0$ in $Q_{2R}(t_0,x_0)$
and
\begin{equation}
\label{eq:measure-ass2-3}
\left| \left\{ v(t,\cdot) \ge H \right\} \cap B_R(x_0) \right| \ge \alpha |B_R(x_0)| \,\,\mbox{for a.e.}\,\, t \in I_{\delta R}^{\ominus}(t_0),
\end{equation}
then there exists a positive constant $C$ depending only on $d,\,s,\, \Lambda,\, q$, such that either
\begin{equation}
\label{Tail-large2}
\left(\fint_{I_{2R}^{\ominus}(t_0)} \operatorname{Tail}(v_-(t);x_0,2R)^q \operatorname{d}\! t \right)^{\frac{1}{q}}>\sigma H,
\end{equation}
or
\begin{equation*}
\left| \left\{ v \le \frac{\sigma H}{4} \right\} \cap I^{\ominus}_{\delta R}(t_0) \times B_R (x_0)\right| \le  \frac{C}{\alpha\delta(\log_2\frac{2}{\sigma})^{\frac{1}{2}}} |I^{\ominus}_{\delta R} (t_0)\times B_R(x_0)|.
\end{equation*}
\end{lemma}
\begin{proof}
Without loss of generality, we may assume that \eqref{Tail-large2} does not hold true. For any fixed $\sigma\in(0,1]$, there exists $j^\ast\in\mathbb{N}_+$ such that
$\frac{1}{2^{j^\ast+1}}\leq \frac{\sigma}{4}<\frac{1}{2^{j^\ast}}$. We denote $k_j:=\frac{H}{2^{j}}$, then it is easy to see that $k_j$ is a decreasing sequence and satisfies $k_j>\frac{\sigma H}{4}$ for $j=0,1,2,...,j^\ast$.
Let $w(t,x) = v(t,x) - k_j$, then
applying the Caccioppoli inequality \eqref{eq:Cacc2} with $l=k_j$ and combining with H\"{o}lder's inequality, we obtain
\begin{eqnarray}\label{lem3-1}
  &&\int_{I_{\delta R}^{\ominus}(t_0)}\int_{B_R(x_0)} |Dw_{-}(t,x)|^2\operatorname{d}\! x\operatorname{d}\! t \nonumber \\
   &\leq&  \int_{B_{2R}(x_0)} w^2_{-}(t_0-(\delta R)^2,x) \operatorname{d}\! x+ C \ R^{-2}\int_{I_{\delta R}^{\ominus}(t_0)} \int_{B_{2R}(x_0)} w^2_{-}(t,x) \operatorname{d}\! x \operatorname{d}\! t \nonumber\\
&&+ C R^{-2} \int_{I_{\delta R}^{\ominus}(t_0)} \left(\int_{B_{2R}(x_0)} w_{-}(t,x) \operatorname{d}\!x \right) \operatorname{Tail}(w_{-}(t);x_0,2R)\operatorname{d}\!t\nonumber\\
&\leq&C k_j^2 (\delta R)^{-2}|I^{\ominus}_{\delta R} \times B_R|.
\end{eqnarray}
We define a set
$$A_j:=\{v\leq k_j\}\cap I_{\delta R}^{\ominus}(t_0)\times B_{R}(x_0), $$
then it is derived from Lemma \ref{auxilem1}, the condition \eqref{eq:measure-ass2-3}, H\"{o}lder's inequality and \eqref{lem3-1} that
\begin{eqnarray*}
  (k_j-k_{j+1})|A_j| &\leq& CR^{d+1}\int_{I_{\delta R}^{\ominus}(t_0)} \frac{\int_{\{k_{j+1}< v(t,\cdot)\leq k_j\}\cap B_R(x_0)}|Dv|\operatorname{d}\!x}{|\{ v(t,\cdot)\geq k_j\}\cap B_R(x_0)|}\operatorname{d}\!t  \\
    &\leq& CR^{d+1}\int_{I_{\delta R}^{\ominus}(t_0)} \frac{\int_{\{k_{j+1}< v(t,\cdot)\leq k_j\}\cap B_R(x_0)}|Dw_-|\operatorname{d}\!x}{|\{ v(t,\cdot)\geq H\}\cap B_R(x_0)|}\operatorname{d}\!t   \\
   &=&\frac{CR^{d+1}}{\alpha|B_R|} \left( \int_{I_{\delta R}^{\ominus}(t_0)} \int_{ B_R(x_0)}|Dw_-|^2\operatorname{d}\!x\operatorname{d}\!t\right)^{\frac{1}{2}}|A_j-A_{j+1}|^{\frac{1}{2}}\\
   &\leq&\frac{C}{\alpha\delta} k_j |I^{\ominus}_{\delta R} \times B_R|^{\frac{1}{2}}|A_j-A_{j+1}|^{\frac{1}{2}}.
\end{eqnarray*}
Note that $k_j-k_{j+1}=\frac{k_j}{2}$ and $|A_j|>|A_{j+1}|$, then summing both sides of the above inequality with respect to $j$ from $0$ to $j^\ast-1$ results in
\begin{equation*}
  |A_{j^\ast}|^2\leq\frac{\displaystyle\sum_{j=0}^{j^\ast-1}|A_{j+1}|^2}{j^\ast}\leq\frac{C}{(\alpha\delta)^2j^\ast}|I^{\ominus}_{\delta R} \times B_R|\displaystyle\sum_{j=0}^{j^\ast-1}|A_j-A_{j+1}|\leq\frac{C}{(\alpha\delta)^2j^\ast}|I^{\ominus}_{\delta R} \times B_R|^2.
\end{equation*}
Finally, the above estimate and $\frac{1}{2^{j^\ast+1}}\leq \frac{\sigma}{4}<\frac{1}{2^{j^\ast}}$ guarantee that $j^\ast\geq \log_2\frac{2}{\sigma}$ and then
\begin{eqnarray*}
  \left| \left\{ v \le \frac{\sigma H}{4} \right\} \cap I^{\ominus}_{\delta R}(t_0) \times B_R (x_0)\right|\leq |A_{j^\ast}|
  \le  \frac{C}{\alpha\delta(\log_2\frac{2}{\sigma})^{\frac{1}{2}}} |I^{\ominus}_{\delta R} (t_0)\times B_R(x_0)|.
\end{eqnarray*}
Thus, we complete the proof of Lemma \ref{lemma:growth-lemma3}\,.
\end{proof}

Let us state the following growth lemma, which is an immediate consequence of Lemma \ref{lemma:growth-lemma1}, Lemma \ref{lemma:growth-lemma2} and Lemma \ref{lemma:growth-lemma3}\,.

\begin{corollary}
\label{cor:growth-lemma}
Let $\alpha \in (0,1]$ and
$v$ be a local weak supersolution to \eqref{eq:PDE-hom} in $\Omega_T$.
For every $R\in(0,1]$ and $(t_0,x_0)\in I\times\Omega$ with $I^\ominus_{2R}(t_0)\times B_{2R}(x_0)\Subset I\times\Omega$, any $q\in(1,2]$  and $H > 0$, if $v \ge 0$ in $Q_{2R}(t_0,x_0)$
and
\begin{equation}
\label{eq:measure-ass2}
\left| \left\{ v(t_1,\cdot) \ge H \right\} \cap B_R (x_0)\right| \ge \alpha |B_R(x_0)| \,\,\mbox{for some}\,\, t_1 \in I_{2 R}^{\ominus}(t_0) ,
\end{equation}
then there exist $\delta,\,\theta \in (0,1)$ depending only on $d,\,s,\,\Lambda,\,q,\,\alpha$ such that
\begin{equation*}
v(t,x) \ge \frac{\theta H}{8} ~~ \text{ in } I_{\frac{\delta R}{2}}^{\ominus}(t_1+(\delta R)^2) \times B_{\frac{R}{2}}(x_0)
\end{equation*}
holds, provided $I_{\delta R}^{\oplus}(t_1)\subset I_{2R}^{\ominus}(t_0)$ and
\begin{equation}\label{Tail-small}
\left(\fint_{I_{2R}^{\ominus}(t_0)} \operatorname{Tail}(v_-(t);x_0,2R)^q \operatorname{d}\! t \right)^{\frac{1}{q}} \leq \theta H.
\end{equation}
\end{corollary}
\begin{proof}
We first choose $\varepsilon,\,\delta\in(0,1)$ depending only on $d,\,s,\, \Lambda,\,q,\,\alpha$ as in Lemma \ref{lemma:growth-lemma2} to yield that
\begin{equation*}
\left| \left\{ v (t,\cdot)\geq \varepsilon H \right\} \cap B_R(x_0) \right| \geq \frac{\alpha}{2}| B_R(x_0)|
\end{equation*}
for almost everywhere $t\in I_{\delta R}^{\oplus}(t_1)= I_{\delta R}^{\ominus}(t_1+(\delta R)^2) \subset I_{2R}^{\ominus}(t_0)$. To proceed, Lemma \ref{lemma:growth-lemma3} shows that there exists a positive constant $C$ depending only on $d,\,s,\, \Lambda,\, q$ such that
\begin{equation*}
\left| \left\{ v \le \frac{\theta\varepsilon H}{4} \right\} \cap I^{\oplus}_{\delta R}(t_1) \times B_R (x_0)\right| \le  \frac{C}{\alpha\delta(\log_2\frac{2}{\theta})^{\frac{1}{2}}} |I^{\oplus}_{\delta R} (t_1)\times B_R(x_0)|.
\end{equation*}
We further select $\theta$ sufficiently small such that $\frac{C}{\alpha\delta(\log_2\frac{2}{\theta})^{\frac{1}{2}}}\leq\lambda$, where the constant $\lambda$ is given in Lemma \ref{lemma:growth-lemma1}.
It follows from the dependence of $\lambda$ that $\theta$ depends only on $d,\,s,\,\Lambda,\,q,\,\alpha$.
Therefore, using Lemma \ref{lemma:growth-lemma1} to derive
\begin{equation*}
v(t,x) \ge \frac{\theta H}{8} ~~ \text{ in } I_{\frac{\delta R}{2}}^{\ominus}(t_1+(\delta R)^2) \times B_{\frac{R}{2}}(x_0),
\end{equation*}
which completes the proof of Corollary \ref{cor:growth-lemma}\,.
\end{proof}

We are now ready to prove the H\"{o}lder estimate with an optimal tail term for the weak solution of the homogeneous mixed local and nonlocal parabolic equation \eqref{eq:PDE-hom}, which is given in Theorem \ref{prop:Holderq}.

\begin{proof}[Proof of Theorem \ref{prop:Holderq}]		
The proof is by constructing two sequences, a non-increasing sequence $\{M_j\}$ and a non-decreasing sequence $\{m_j\}$, which satisfy
\begin{equation}
\label{eq:osc-decay}
m_j \le v \le M_j ~~ \text{ in } Q_{\nu^{-j} R}(t_0,x_0), ~~ \text{ and }~~ M_j - m_j = L \nu^{-\gamma_0 j},
\end{equation}
for any $j \in \mathbb{N}$, some small $\gamma_0 \in (0,1)$ and large $\nu > 1$,
where
\begin{equation*}
L := C_5 \Vert v \Vert_{L^{\infty}(Q_R(t_0,x_0))} + \left( \fint_{I_{2R}^{\ominus}(t_0)} \operatorname{Tail}(v(t);x_0,2R)^q \operatorname{d}\! t  \right)^{\frac{1}{q}}
\end{equation*}
for some positive constant $C_5$ to be determined later. Once \eqref{eq:osc-decay} is achieved, we immediately obtain the desired H\"{o}lder continuity by using the definition of $L$ and the boundedness of $v$ established in Lemma \ref{lemma:locbd}\,.

For any fixed $j_0$, if we set $M_j = \frac{\nu^{-\gamma_0 j}L}{2}$ and $m_j = - \frac{\nu^{-\gamma_0 j}L}{2}$, and choose $C_5 \ge 2 \nu^{\gamma_0 j_0}$, then it is not difficult to verify that \eqref{eq:osc-decay} is valid for every $j\leq j_0$.
Next, we adopt induction to prove that \eqref{eq:osc-decay} holds true for $j>j_0$.
More precisely, assuming that \eqref{eq:osc-decay} is valid up to some $j\geq j_0$, we will show that \eqref{eq:osc-decay} holds true for $j+1$.
To proceed, we need to construct appropriate $M_{j+1}$ and $m_{j+1}$.
Without loss of generality, we suppose that
\begin{equation}
\label{holder1}
\big|\{ v(t_0-(\delta\nu^{-j}R)^2 ,\cdot) \ge m_j + \frac{M_j - m_j}{2} \} \cap B_{\nu^{-j} R}(x_0)\big| \ge \frac{1}{2}|B_{\nu^{-j}R}(x_0)|.
\end{equation}
Through a straightforward calculation, we can verify that
\begin{equation*}
\left(\fint_{I_{2\nu^{-j}R}^{\ominus}(t_0)} \operatorname{Tail}((v-m_j)_-(t);x_0,2\nu^{-j}R)^q \operatorname{d}\! t \right)^{\frac{1}{q}} \leq \frac{\theta (M_j - m_j)}{2}
\end{equation*}
holds true by first choosing a sufficiently small $\gamma_0\in(0,1)$, followed by selecting a sufficiently large $\nu>\frac{2}{\delta}$, both of which  depend only on $d,\,s,\,q,\,\theta$.
Then it follows from Corollary \ref{cor:growth-lemma} with $v:=v-m_j$, $R:=\nu^{-j} R$, $\alpha=\frac{1}{2}$, $t_1=t_0-(\delta\nu^{-j}R)^2$ and $H:= \frac{M_j - m_j}{2}=\frac{L\nu^{-\gamma_0 j}}{2}$ that
\begin{equation}\label{holder2}
v(t,x) \ge m_j+\frac{\theta L\nu^{-\gamma_0 j}}{16} ~~ \text{ in } Q_{\nu^{-(j+1)}R}(t_0,x_0),
\end{equation}
where the constant $\theta=\theta(d,s,\Lambda,q)\in(0,1)$. Thus, we define
\begin{equation*}
  M_{j+1}=M_j \,\,\mbox{and}\,\, m_{j+1}=M_j-L\nu^{-\gamma_0(j+1)},
\end{equation*}
and further choose $\gamma_0\leq\log_\nu(\frac{16}{16-\theta})$ sufficiently small such that the chosen $M_{j+1}$ and $m_{j+1}$ satisfy \eqref{eq:osc-decay}. On the other hand, if \eqref{holder1} fails, we proceed analogously, but instead apply Corollary \ref{cor:growth-lemma} with $v: = M_j - v$. In this case, let $M_{j+1}=L\nu^{-\gamma_0(j+1)}+m_j$ and $m_{j+1}=m_j$ as desired, then we complete the proof of Theorem \ref{prop:Holderq}\,.
\end{proof}

\section{Excess decay estimate}\label{section4}

This section is devoted to proving an excess decay estimate for the mixed local and nonlocal parabolic equation \eqref{eq1}, which is a crucial ingredient in the proof of our Riesz potential results. In this context, we introduce the following global excess functional to encode the information about the oscillations of solutions, taking into account simultaneously the behaviour in small parabolic cylinders and the long-range interactions
\begin{eqnarray*}
  E(u,t_0,x_0,R) &:=& \left[\fint_{{I^\ominus_{R}(t_0)}}\left(\fint_{{B_{R}(x_0)}}  |u - (u)_{Q_{R}(t_0,x_0)} | \operatorname{d}\!x\right)^q\operatorname{d}\!t\right]^{\frac{1}{q}}\\
   &&+ \left(\fint_{{I^\ominus_{R}(t_0)}}
   \operatorname{Tail}(u(t)-(u)_{Q_{R}(t_0,x_0)};x_0,R)^q\operatorname{d}\!t\right)^{\frac{1}{q}},
\end{eqnarray*}
where the symbol $(u)_{Q_{R}(t_0,x_0)}$ denotes the integral average of a locally integrable function $u$ over the parabolic cylinder $Q_{R}(t_0,x_0)$,
defined as
\begin{equation*}
  (u)_{Q_{R}(t_0,x_0)}:=\fint_{I^\ominus_{R}(t_0)}\fint_{B_{R}(x_0)}u(t,x)\operatorname{d}\!x\operatorname{d}\!t
  =\frac{1}{|Q_R|}\int_{I^\ominus_{R}(t_0)}\int_{B_{R}(x_0)}u(t,x)\operatorname{d}\!x\operatorname{d}\!t.
\end{equation*}

We first need to establish a comparison estimate between the weak solution of
\begin{equation}\label{localeq}
\partial_t u-\operatorname{div}\left(\mathcal{A}(t,x, Du)\right)  +\mathcal{L} u = \mu ~~\text{in}~~Q_{2R}(t_0,x_0),
\end{equation}
for any $R>0$ and $(t_0,x_0)\in I\times\Omega$ with $Q_{2R}(t_0,x_0)\Subset I\times\Omega$,
and that of the desired problem
\begin{equation}\label{Comparisoneq}
\left\{\begin{array}{r@{\ \ }c@{\ \ }ll}
\partial_t v -\operatorname{div}\left(\mathcal{A}(t,x, Dv)\right) +\mathcal L v & =& 0 & \mbox{in}\ \ Q_{R}(t_0,x_0)\,, \\[0.05cm]
v&=& u & \mbox{in}\ \ I_R^\ominus(t_0)\times (\mathbb{R}^d\backslash B_{R}(x_{0}))\,, \\[0.05cm]
v&=&u & \mbox{in}\ \ \{t_0-R^2\}\times B_{R}(x_{0})\,,
\end{array}\right.
\end{equation}
for which we have known H\"{o}lder regularity result.

\begin{lemma}\label{thm:comparison}
Let $u$ and $v$ be weak solutions to \eqref{localeq} and \eqref{Comparisoneq} respectively, then there holds that
\begin{equation}\label{com-estimate}
	\sup_{t\in I_R^\ominus(t_0)}\fint_{B_R(x_0)} |u(t,x)-v(t,x)|\operatorname{d}\! x \leq \frac{2|\mu|(Q_R(t_0,x_0))}{|B_R|}.
	\end{equation}
\end{lemma}
\begin{proof}
The proof is done by applying the truncation method.
Let $w(t,x):=u(t,x)-v(t,x)$,
we first select $\phi_{\epsilon}^{\pm} = \pm \left(1 \wedge \frac{w_{\pm}}{\epsilon} \right)$ as a test function for $w$ with respect to the spatial variable, then it yields that
\begin{eqnarray*}
 && \int_{B_R(x_0)} (\partial_t w) \phi^{\pm}_{\epsilon} \operatorname{d}\! x + \int_{B_R(x_0)} \left(\mathcal{A}(t,x,Du)-\mathcal{A}(t,x,Dv)\right)\cdot D\phi^{\pm}_{\epsilon}\operatorname{d}\! x\\
 &&+ \int_{\mathbb{R}^d} \int_{\mathbb{R}^d} (w(t,x) - w(t,y))(\phi^{\pm}_{\epsilon}(t,x) - \phi^{\pm}_{\epsilon}(t,y)) K(t,x,y) \operatorname{d}\! x \operatorname{d}\! y\\
   &=&  \int_{B_R(x_0)} \phi_{\epsilon}^{\pm} \operatorname{d}\!\mu \leq |\mu(t)|(B_R(x_0)),
\end{eqnarray*}
where last inequality is guaranteed by $|\phi_{\epsilon}^{\pm}| \leq 1$. Utilizing the ellipticity condition \eqref{ellipticity} of $\mathcal{A}$ and the definition of $\phi_{\epsilon}^{\pm}$, we obtain
$$\left(\mathcal{A}(t,x,Du)-\mathcal{A}(t,x,Dv)\right)\cdot D\phi^{\pm}_{\epsilon}\geq0\,\,\mbox{and}\,\,(w(t,x) - w(t,y))(\phi^{\pm}_{\epsilon}(t,x) - \phi^{\pm}_{\epsilon}(t,y)) \ge 0.$$
Then it follows that
\begin{equation}\label{lemu-v-equ2}
  \int_{B_R(x_0)} (\partial_t w) \phi^{\pm}_{\epsilon} \operatorname{d}\! x\leq |\mu(t)|(B_R(x_0)).
\end{equation}
Next, for any given $\hat{t} \in I_R^{\ominus}(t_0)$ and $\delta > 0$ with $\hat{t}+\delta<t_0$, we define the test function for $w$ with respect to the time variable as
\begin{equation*}
\eta_{\delta}(t) =
\left\{\begin{array}{r@{\ \ }c@{\ \ }ll}
1, & t< \hat{t}\,, \\[0.05cm]
1-\frac{t-\hat{t}}{\delta}, & \hat{t}\leq t\leq \hat{t}+\delta\,, \\[0.05cm]
0, & t> \hat{t}+\delta\,.
\end{array}\right.
\end{equation*}
It is obvious that $\eta_{\delta}(t_0)=0$ and $\eta_{\delta}\in[0,1]$.
A combination of integration by parts and $w(t_0-R^2,x)=0$ in $B_R(x_0)$ yields  that			
\begin{eqnarray*}
\int_{I_R^\ominus(t_0)}\eta_{\delta}(t)\int_{B_R(x_0)} (\partial_t w) \phi^{\pm}_{\epsilon} \operatorname{d}\! x\operatorname{d}\! t&=&
  \int_{I_R^\ominus(t_0)}\int_{B_R(x_0)}\eta_\delta(t)\partial_t\left( \int_0^{w_{\pm}} \left( 1 \wedge \frac{\sigma}{\epsilon} \right)  \operatorname{d}\!\sigma \right)\operatorname{d}\! x\operatorname{d}\! t
\\
   &=&   -\int_{I_R^\ominus(t_0)}\partial_t \eta_\delta(t)\int_{B_R(x_0)} \int_0^{w_{\pm}} \left( 1 \wedge \frac{\sigma}{\epsilon} \right)  \operatorname{d}\!\sigma \operatorname{d}\! x\operatorname{d}\! t.\\
\end{eqnarray*}
Then by virtue of \eqref{lemu-v-equ2}, we derive
\begin{equation*}
   -\int_{I_R^\ominus(t_0)}\partial_t \eta_\delta(t)\int_{B_R(x_0)} \int_0^{w_{\pm}} \left( 1 \wedge \frac{\sigma}{\epsilon} \right)  \operatorname{d}\!\sigma \operatorname{d}\! x\operatorname{d}\! t\leq \int_{I_R^\ominus(t_0)}\eta_\delta(t)|\mu(t)|(B_r(x_0))\operatorname{d}\! t\leq|\mu|(Q_R(t_0,x_0)).
\end{equation*}
We first take the limit $\epsilon\rightarrow 0$ of the above inequality, then it follows from the dominated convergence theorem that
\begin{eqnarray*}
   \int_{I_R^\ominus(t_0)}\left(-\partial_t \eta_\delta(t)\right)\int_{B_R(x_0)} w_{\pm}(t,x) \operatorname{d}\! x\operatorname{d}\! t\leq|\mu|(Q_R(t_0,x_0)).
\end{eqnarray*}
We proceed by letting $\delta \rightarrow 0$, since
$- \partial_t\eta_{\delta} \to \delta_{\hat{t}}$ as $\delta \rightarrow 0$, then it follows from Levi's theorem that
\begin{equation*}
 \fint_{B_R(x_0)}w_{\pm}(\hat{t},x)\operatorname{d}\!x\leq \frac{|\mu|(Q_R(t_0,x_0))}{|B_R|}
\end{equation*}
for any $\hat{t} \in I_R^{\ominus}(t_0)$. By the arbitrariness of $\hat t$, we deduce that
\begin{equation}\label{lemu-v-equ3}
 \sup_{t\in I_R^\ominus(t_0)}\fint_{B_R(t_0,x_0)}w_{\pm}(t,x)\operatorname{d}\!x\operatorname{d}\!t\leq \frac{|\mu|(Q_R(t_0,x_0))}{|B_R|}.
\end{equation}
Finally, by summing \eqref{lemu-v-equ3} for the positive and negative parts, we arrive at the desired comparison estimate \eqref{com-estimate}. Hence, we complete the proof of Lemma \ref{thm:comparison}\,.
\end{proof}

With the aid of the H\"older continuity for the solution $v$ of the homogeneous mixed parabolic equation \eqref{eq:PDE-hom} obtained in Theorem \ref{prop:Holderq} and the comparison estimate \eqref{com-estimate} between the solution $u$ of the nonhomogeneous mixed parabolic equation \eqref{eq1} and $v$, we are now able to derive the excess decay estimate for $u$.
\begin{lemma} \label{lemma:OscDecu}
Let $u$ be a local weak solution to \eqref{eq1} with $\mu \in \mathcal{M}(\mathbb{R}^{d+1})$ in $\Omega_T$,
then for any $q\in(1,2]$, any integer $m \geq 1$, any $R\in(0,1]$ and $(t_0,x_0)\in I\times\Omega$ with $I^\ominus_{2R}(t_0)\times B_{2R}(x_0)\Subset I\times\Omega$, there exists a positive constant $C_0$ depending only on $d,\,s,\,\Lambda,\,q$ such that
\begin{equation} \label{eq:oscdecay}
E(u,t_0,x_0,2^{-m} R)  \leq C_0 2^{-\gamma_0 m} E(u,t_0,x_0,R) + C_0 2^{2m(\frac{1}{q}+s+d)}R^{-d} |\mu|(Q_{R}(t_0,x_0)),
\end{equation}
where the constant $\gamma_0=\gamma_0(d,s,\Lambda,q)$ is the H\"{o}lder continuity exponent for the homogeneous mixed parabolic equation \eqref{eq:PDE-hom} given by Theorem \ref{prop:Holderq}\,.
\end{lemma}

\begin{proof}
Let $v$ be a weak solution of the comparison equation \eqref{Comparisoneq}, we first establish an excess decay estimate for $v$.
By performing a direct calculation and utilizing H\"older's inequality and  the H\"older estimate of $v$ established in Theorem \ref{prop:Holderq}\,, we derive
\begin{eqnarray}\label{decayv}
&&E(v,t_0,x_0,2^{-m}R)\nonumber\\  &\leq& \left[\fint_{ I^{\ominus}_{2^{-m}R}(t_0)} \left(\fint_{B_{2^{-m}R}(x_0)} |v(t,x) - (v)_{Q_{2^{-m}R}(t_0,x_0)} | \operatorname{d}\!x \right)^q\operatorname{d}\!t\right]^{\frac{1}{q}}  \nonumber\\
   &&  + (2^{-m}R)^{2} \sum_{k=1}^{m}\left[\fint_{ I^{\ominus}_{2^{-m}R}(t_0)} \left( \int_{B_{2^{k-m} R}(x_0) \setminus B_{2^{k-m-1} R}(x_0)} \frac{|v - (v)_{Q_{2^{-m} R}(t_0,x_0)}|}{|x_0-x|^{d+2s}}\operatorname{d}\! x\right)^q\operatorname{d}\!t\right]^{\frac{1}{q}} \nonumber\\
  && + (2^{-m}R)^{2}\left[\fint_{ I^{\ominus}_{2^{-m}R}(t_0)} \left( \int_{\mathbb{R}^d \setminus B_{R}(x_0)} \frac{|v - (v)_{Q_{2^{-m} R}(t_0,x_0)}|}{|x_0-x|^{d+2s}}\operatorname{d}\! x\right)^q\operatorname{d}\!t\right]^{\frac{1}{q}} \nonumber\\
 &\leq& C \sum_{k=0}^{m-1} 2^{-2sk} \left[\fint_{ I^{\ominus}_{2^{-m}R}(t_0)} \left( \fint_{B_{2^{k-m}R}(x_0)} |v(t,x) - (v)_{Q_{2^{-m}R}(t_0,x_0)} | \operatorname{d}\!x\right)^q\operatorname{d}\!t\right]^{\frac{1}{q}}   \nonumber\\
 && +2^{-2m(1-\frac{1}{q})} \left[\fint_{ I^{\ominus}_{R}(t_0)} \left( \fint_{B_{R}(x_0)} |v(t,x) - (v)_{Q_{2^{-m}R}(t_0,x_0)} | \operatorname{d}\!x\right)^q\operatorname{d}\!t\right]^{\frac{1}{q}}   \nonumber\\
 && +2^{-2m(1-\frac{1}{q})} \left(\fint_{{I^\ominus_{R}(t_0)}}
   \operatorname{Tail}(v(t)-(v)_{Q_{R}(t_0,x_0)};x_0,R)^q\operatorname{d}\!t\right)^{\frac{1}{q}} \nonumber\\
  &\leq& C\sum_{k=0}^{m-1} 2^{-2sk}\underset{{Q_{2^{k-m}R}(t_0,x_0)}}{\operatorname{osc}} v+C\sum_{k=0}^{m-1} 2^{-2(1-\frac{1}{q})k}\underset{{Q_{2^{k-m}R}(t_0,x_0)}}{\operatorname{osc}} v+C\left(2^{-2sm}+2^{-2m(1-\frac{1}{q})}\right) E(v,t_0,x_0,R)  \nonumber\\
  & \leq& C 2^{-\gamma_0 m} E(v,t_0,x_0,R)\left(\sum_{k=0}^{\infty} 2^{(\gamma_0-2s)k}+\sum_{k=0}^{\infty} 2^{(\gamma_0-2(1-\frac{1}{q}))k} \right) +C2^{-\gamma_0 m} E(v,t_0,x_0,R)\nonumber\\
  & \leq&C 2^{-\gamma_0 m}E(v,t_0,x_0,R),
\end{eqnarray}
where the constant $C=C(d,s,\Lambda,q)$, and the H\"older continuity exponent $\gamma_0$ of Theorem \ref{prop:Holderq} may assume that $\gamma_0<\min\{2s,2(1-\frac{1}{q})\}$, which ensures that the last inequality is true.

We now turn to the proof of the excess decay estimate for $u$.
Note that
$$v=u\,\,\mbox{in}\,\,I_R^\ominus(t_0)\times (\mathbb{R}^d\setminus B_{R}(x_{0})),$$
and combining \eqref{decayv} with the comparison estimate given in Lemma \ref{thm:comparison}, we conclude that
\begin{eqnarray*}
&&E(u,t_0,x_0,2^{-m}R)\\
& \leq& E(v,t_0,x_0,2^{-m}R) + E(u-v,t_0,x_0,2^{-m} R) \\
		& \leq &C 2^{-\gamma_0 m} E(v,t_0,x_0,R) +C2^{2m(\frac{1}{q}+s+d)} E(u-v,t_0,x_0,R)  \nonumber\\
&\leq&C 2^{-\gamma_0 m} E(v,t_0,x_0,R) +C2^{2m(\frac{1}{q}+s+d)}\sup_{t\in I_R^\ominus(t_0)}\fint_{B_R(x_0)} |u(t,x)-v(t,x)|\operatorname{d}\! x \\
		& \leq &C_0 2^{-\gamma_0 m} E(u,t_0,x_0,R) + C_0 2^{2m(\frac{1}{q}+s+d)} R^{-d} |\mu|(Q_{R}(t_0,x_0)),
\end{eqnarray*}
where the positive constant $C_0$ depends only on $d,\,s,\,\Lambda,\,q$.
Therefore, we conclude that the proof of Lemma \ref{lemma:OscDecu} is complete.
\end{proof}

\section{Riesz potential estimates}\label{section5}

 We are now in a position to prove our main results concerning the pointwise estimate and the oscillation estimate for the weak solution of the mixed local and nonlocal parabolic equation \eqref{eq1} via caloric Riesz potentials.
 These results are given by Theorem \ref{thm:PE} and Theorem \ref{thm:OPE}, respectively.

\subsection{Pointwise estimate}

In this subsection, we devoted ourselves to the proof of Theorem \ref{thm:PE}\,, which relates to the zero-order pointwise Riesz potential estimate.
\begin{proof}[Proof of Theorem \ref{thm:PE}]
We choose $m \geq 1$ sufficiently large such that
\begin{equation} \label{llarge}
	C_0 2^{-\gamma_0 m} \leq \frac{1}{2},
\end{equation}
then it is evident that $m$ only depends on $d,\,s,\,\Lambda,\,q$.
For any $i\in \mathbb{N}$, we denote
$$	E_i:=E(u,t_0,x_0,2^{-im} R),$$
then a combination of Lemma \ref{lemma:OscDecu} with \eqref{llarge} yields that
\begin{equation} \label{pe1}
E_{i+1}  \leq \frac{1}{2} E_{i} + C_02^{2m(\frac{1}{q}+s+d)}  (2^{-im}R)^{-d} |\mu|(Q_{2^{-im} R}(t_0,x_0)).
\end{equation}
For $l\in \mathbb{N}_+$, summing \eqref{pe1} over $i \in \{0,...,l-1\}$ results in
\begin{equation*}
\sum_{i=1}^{l} E_{i}  \leq \frac{1}{2} \sum_{i=0}^{l-1} E_{i} + C_0 2^{2m(\frac{1}{q}+s+d)} \sum_{i=0}^{l-1} (2^{-im}R)^{-d} |\mu|(Q_{2^{-im} R}(t_0,x_0)).
\end{equation*}
Note that the first term on the right side of the above inequality can be absorbed by the left side, then we obtain
\begin{equation*}
\sum_{i=1}^{l} E_{i}  \leq 2 E_0 + 2C_0 2^{2m(\frac{1}{q}+s+d)} \sum_{i=0}^{\infty} (2^{-im}R)^{-d} |\mu|(Q_{2^{-im} R}(t_0,x_0)).
\end{equation*}
Combining with H\"older's inequality, we deduce that
\begin{eqnarray}\label{pe2}
  &&  |(u)_{Q_{2^{-(l+1)m}R}(t_0,x_0)}| \nonumber\\
   &\leq&  \sum_{i=0}^l (|(u)_{Q_{2^{-(i+1)m}R}(t_0,x_0)}-(u)_{Q_{2^{-im}R}(t_0,x_0)}|)
   +|(u)_{Q_{R}(t_0,x_0)}|\nonumber\\
   &\leq&  C\sum_{i=0}^l \fint_{Q_{2^{-im}R}(t_0,x_0)}|u(t,x)-(u)_{Q_{2^{-im}R}(t_0,x_0)}| \operatorname{d}\!x\operatorname{d}\!t
   +\fint_{Q_{R}(t_0,x_0)}|u(t,x)|\operatorname{d}\!x\operatorname{d}\!t\nonumber\\
   &\leq&  C\sum_{i=1}^l E_i
   +C\fint_{Q_{R}(t_0,x_0)}|u(t,x)|\operatorname{d}\!x\operatorname{d}\!t\nonumber\\
   &\leq& C\left[\left(\fint_{Q_{R}(t_0,x_0)} |u(t,x)| ^q \operatorname{d}\!x\operatorname{d}\!t \right)^{\frac{1}{q}}+  \left(\fint_{{I^\ominus_{R}(t_0)}}
   \operatorname{Tail}(u(t);x_0,R)^q\operatorname{d}\!t\right)^{\frac{1}{q}}\right.\nonumber\\
&& \left.
   +\sum_{i=0}^{\infty} (2^{-im}R)^{-d} |\mu|(Q_{2^{-im} R}(t_0,x_0))\right],
\end{eqnarray}
where the positive constant $C$ depends only on $d,\,s,\,\Lambda,\,q$. Furthermore, it remains to estimate the last term in \eqref{pe2}. Let $H:=2^m$, then it follows that
\begin{eqnarray}\label{pe3}
   && \sum_{i=0}^{\infty} (H^{-i}R)^{-d} |\mu|(Q_{H^{-i} R}(t_0,x_0)) \nonumber\\
  &=&  R^{-d} |\mu|(Q_{R}(t_0,x_0))+\sum_{i=0}^{\infty} (H^{-i-1}R)^{-d} |\mu|(Q_{H^{-i-1} R}(t_0,x_0)) \nonumber\\
  &\leq& \frac{1}{\ln 2}\int_{R}^{2R} \frac{|\mu|(Q_{R}(t_0,x_0))}{R^{d}}
  \frac{\operatorname{d}\!\rho}{\rho}+\frac{1}{\ln H  }\sum_{i=0}^{\infty}\int_{H^{-i-1}R}^{H^{-i}R}(H^{-i-1}R)^{-d} |\mu|(Q_{H^{-i-1} R}(t_0,x_0))
  \frac{\operatorname{d}\!\rho}{\rho}\nonumber \\
  &\leq& \frac{2^{d}}{\ln 2}\int_{R}^{2R} \frac{|\mu|(Q_{\rho}(t_0,x_0))}{\rho^{d}}
  \frac{\operatorname{d}\!\rho}{\rho}+\frac{H^{d}}{\ln H  }\sum_{i=0}^{\infty}\int_{H^{-i-1}R}^{H^{-i}R}\frac{|\mu|(Q_{\rho}(t_0,x_0))}{\rho^{d}}
  \frac{\operatorname{d}\!\rho}{\rho}\nonumber \\
  &\leq& C(d,s,\Lambda,q)\int_{0}^{2R} \frac{|\mu|(Q_{\rho}(t_0,x_0))}{\rho^{d}}
  \frac{\operatorname{d}\!\rho}{\rho}=C\mathcal {I}^{\mu}_{2}(t_0,x_0;2R).
\end{eqnarray}
Finally, substituting \eqref{pe3} into \eqref{pe2} and letting
$l\rightarrow\infty$, then we conclude from the Lebesgue differentiation theorem that the pointwise Riesz potential estimate \eqref{eq:PE} holds true. Hence, the proof of Theorem \ref{thm:PE} is complete.
\end{proof}

\subsection{Oscillation estimate}

In the last subsection, we prove the oscillation estimate given in Theorem \ref{thm:OPE} by introducing two appropriate fractional maximal operators.
\begin{proof}[Proof of Theorem \ref{thm:OPE}]
At this stage, we shall use two fractional maximal operators, one being the restricted fractional maximal function of order $2-\gamma$ for $\mu$, defined by
$${\bf M}_{2-\gamma, R}(\mu)(t,x):=\sup_{0<r\leq R} r^{2-\gamma}\frac{|\mu|(Q_r(t,x))}{r^{d+2}},$$
the other is the restricted nonlocal fractional sharp maximal function of order $\gamma$ for $u$, defined by
$${\bf M}^{\sharp}_{\gamma, R}(u)(t,x):=\sup_{0<r\leq R}r^{-\gamma} E(u,t,x,r).$$
Let us first claim that there exists a positive constant $C=C(d,s,\Lambda,q,\tilde{\gamma})$ such that
\begin{equation}\label{ope-1}
  {\bf M}^{\sharp}_{\gamma, R}(u)(t,x)\leq C{\bf M}_{2-\gamma, R}(\mu)(t,x)+CR^{-\gamma}E(u,t,x,R)
\end{equation}
holds for any $\gamma\in [0,\tilde{\gamma}]$ and any parabolic cylinder $Q_{2R}(t,x)\Subset I\times \Omega$.

For any fixed $\rho\in (0,R]$, and $\sigma\in(0,\frac{1}{2})$ to be determined later, then Lemma \ref{lemma:OscDecu} implies that
\begin{equation*}
  (\sigma\rho)^{-\gamma}E(u,t,x,\sigma\rho)\leq C_0(d,s,\Lambda,q) \sigma^{\gamma_0-\gamma}\rho^{-\gamma}E(u,t,x,\rho)+C(d,s,\Lambda,q,\sigma) \rho^{-d-\gamma}|\mu|(Q_\rho(t,x)).
\end{equation*}
Selecting $\sigma=\sigma(d,s,\Lambda,q,\tilde{\gamma})$ small enough such that $C_0\sigma^{\gamma_0-\tilde{\gamma}}=\frac{1}{2}$, then $C_0\sigma^{\gamma_0-\gamma}\leq\frac{1}{2}$ ensured by $\gamma\leq \tilde{\gamma}$. It follows that
\begin{eqnarray*}
  (\sigma\rho)^{-\gamma}E(u,t,x,\sigma\rho)&\leq&\frac{1}{2}\rho^{-\gamma}E(u,t,x,\rho)+C \rho^{2-\gamma}\frac{|\mu|(Q_\rho(t,x))}{\rho^{d+2}}\\
  &\leq&\frac{1}{2}{\bf M}^{\sharp}_{\gamma, R}(u)(t,x)+C{\bf M}_{2-\gamma, R}(\mu)(t,x),
\end{eqnarray*}
where the positive constant $C=C(d,s,\Lambda,q,\tilde{\gamma})$. Due to the arbitrariness of $\rho\in (0,R]$, we further obtain
\begin{equation}\label{ope-2}
  \sup_{0<r\leq \sigma R} r^{-\gamma}E(u,t,x,r)\leq\frac{1}{2}{\bf M}^{\sharp}_{\gamma, R}(u)(t,x)+C{\bf M}_{2-\gamma, R}(\mu)(t,x).
\end{equation}
Meanwhile, by a direct calculation, we can derive
\begin{equation}\label{ope-3}
  \sup_{\sigma R<r\leq R} r^{-\gamma}E(u,t,x,r)\leq C(d,s,\Lambda,q,\tilde{\gamma})R^{-\gamma}E(u,t,x,R).
\end{equation}
A combination of \eqref{ope-2} and \eqref{ope-3} yields that
\begin{equation*}
  {\bf M}^{\sharp}_{\gamma, R}(u)(t,x)\leq\frac{1}{2}{\bf M}^{\sharp}_{\gamma, R}(u)(t,x)+C{\bf M}_{2-\gamma, R}(\mu)(t,x)+CR^{-\gamma}E(u,t,x,R),
\end{equation*}
where the first term on the right side can be absorbed by the left side, and thus the assertion \eqref{ope-1} is valid.

Let $(t,x),\,(\tau,y)\in Q_{\frac{R}{8}}(t_0,x_0)$ be Lebesgue points of $u$ and $r\in(0,\frac{R}{8}]$, then for any $0<\rho\leq \frac{r}{8}$, there exist $k\in\mathbb{N}$ and $\theta\in(\frac{1}{4},\frac{1}{2}]$ such that $\rho=\theta^k\frac{r}{4}$.
A direct calculation shows that
\begin{eqnarray}\label{ope-4}
  |(u)_{Q_{\frac{r}{4}}(t,x)}-(u)_{Q_\rho(t,x)}| &\leq&\sum_{i=0}^{k-1}|(u)_{Q_{\theta^i\frac{r}{4}}(t,x)}-(u)_{Q_{\theta^{i+1}\frac{r}{4}}(t,x)}|  \nonumber  \\
   &\leq& \theta^{-\frac{2}{q}-d}\sum_{i=0}^{k-1}E(u,t,x, \theta^i\frac{r}{4})\nonumber \\
    &=&\frac{\theta^{-\frac{2}{q}-d}}{\ln(\theta^{-1})}\sum_{i=0}^{k-1}\int_{\theta^i \frac{r}{4}}^{\theta^{i-1} \frac{r}{4}}E(u,t,x, \theta^i\frac{r}{4})\frac{\operatorname{d}\!\varrho }{\varrho}\nonumber \\
     &\leq&C(d,s,q)\frac{4^{\frac{2}{q}+d}}{\ln2}\sum_{i=0}^{k-1}\int_{\theta^i \frac{r}{4}}^{\theta^{i-1} \frac{r}{4}}E(u,t,x, \varrho)\frac{\operatorname{d}\!\varrho }{\varrho}\nonumber \\
     &\leq&C(d,s,q)\int_{\rho}^{ r}E(u,t,x, \varrho)\frac{\operatorname{d}\!\varrho }{\varrho}.
\end{eqnarray}
To estimate the last term, we apply Lemma \ref{lemma:OscDecu} again, together with the choice of $\sigma$, to deduce that
\begin{eqnarray*}
 \int_{\rho}^{ r}E(u,t,x, \varrho)\frac{\operatorname{d}\!\varrho }{\varrho}  &\leq& \int_{\sigma\rho}^{ r}E(u,t,x, \varrho)\frac{\operatorname{d}\!\varrho }{\varrho} \\
   &=&\int_{\rho}^{r}E(u,t,x, \sigma\varrho)\frac{\operatorname{d}\!\varrho }{\varrho}+\int_{\sigma r}^{ r}E(u,t,x, \varrho)\frac{\operatorname{d}\!\varrho }{\varrho}  \\
   &\leq&\frac{1}{2}\int_{\rho}^{ r}E(u,t,x, \varrho)\frac{\operatorname{d}\!\varrho }{\varrho}+C\int_{\rho}^{r} \frac{|\mu|(Q_{\varrho}(t,x))}{\varrho^d}\frac{\operatorname{d}\!\varrho }{\varrho}+CE(u,t,x, r)  \\
   &\leq&2C\int_{\rho}^{r} \frac{|\mu|(Q_{\varrho}(t,x))}{\varrho^d}\frac{\operatorname{d}\!\varrho }{\varrho}+2CE(u,t,x, r)\\
   &\leq&Cr^\gamma\mathcal {I}^\mu_{2-\gamma}(t,x;R) +CE(u,t,x, r),
\end{eqnarray*}
where the positive constant $C=C(d,s,\Lambda,q,\tilde{\gamma})$. Substituting the above inequality into \eqref{ope-4}, we derive
\begin{equation*}
  |(u)_{Q_{\frac{r}{4}}(t,x)}-(u)_{Q_\rho(t,x)}|\leq Cr^\gamma\mathcal {I}^\mu_{2-\gamma}(t,x;R) +CE(u,t,x, r).
\end{equation*}
Letting $\rho\rightarrow 0$, then it follows from the Lebesgue differentiation theorem that
\begin{equation*}
  |(u)_{Q_{\frac{r}{4}}(t,x)}-u(t,x)|\leq Cr^\gamma\mathcal {I}^\mu_{2-\gamma}(t,x;R) +CE(u,t,x, r).
\end{equation*}
Treating in a completely similar manner, we also have
\begin{equation*}
  |(u)_{Q_{\frac{r}{4}}(\tau,y)}-u(\tau,y)|\leq Cr^\gamma\mathcal {I}^\mu_{2-\gamma}(\tau,y;R) +CE(u,\tau,y, r).
\end{equation*}
We now take $r=\frac{|t-\tau|^{\frac{1}{2}}+|x-y|}{3}<\frac{R}{8}$, then it is easy to verify that
$Q_r(\tau,y)\subset Q_{4r}(t,x)$. Combining the last two inequalities with \eqref{ope-1}, the fact that $Q_{\frac{R}{2}}(t,x)\subset Q_{R}(t_0,x_0)$, H\"{o}lder's inequality and Lemma 4.1 presented in \cite{KuMi}, we conclude that
\begin{eqnarray*}
&&|u(t,x)-u(\tau,y)| \nonumber\\
&\leq&|(u)_{Q_{\frac{r}{4}}(t,x)}-(u)_{Q_{\frac{r}{4}}(\tau,y)}|+CE(u,t,x, r)+CE(u,\tau,y, r)+Cr^\gamma\left[\mathcal {I}^\mu_{2-\gamma}(t,x;R)+\mathcal {I}^\mu_{2-\gamma}(\tau,y;R)\right]\\
&\leq&CE(u,t,x, 4r)+Cr^\gamma\left[\mathcal {I}^\mu_{2-\gamma}(t,x;R)+\mathcal {I}^\mu_{2-\gamma}(\tau,y;R)\right]\\
&\leq& Cr^\gamma{\bf M}^{\sharp}_{\gamma, \frac{R}{2}}(u)(t,x)+Cr^\gamma\left[\mathcal {I}^\mu_{2-\gamma}(t,x;R)+\mathcal {I}^\mu_{2-\gamma}(\tau,y;R)\right]\\
 &\leq&Cr^\gamma\left[{\bf M}_{2-\gamma,\frac{R}{2}}(\mu)(t,x)+R^{-\gamma}E(u,t,x,\frac{R}{2})\right]+Cr^\gamma\left[\mathcal {I}^\mu_{2-\gamma}(t,x;R)+\mathcal {I}^\mu_{2-\gamma}(\tau,y;R)\right]\\
 &\leq&C(\frac{r}{R})^\gamma E(u,t_0,x_0,R)+Cr^\gamma\left[\mathcal {I}^\mu_{2-\gamma}(t,x;R)+\mathcal {I}^\mu_{2-\gamma}(\tau,y;R)\right]\\
 &\leq& C
\left[\left(\fint_{Q_{R}(t_0,x_0)} |u| ^q \operatorname{d}\!x\operatorname{d}\!t \right)^{\frac{1}{q}}+  \left(\fint_{{I^\ominus_{R}(t_0)}}
   \operatorname{Tail}(u(t);x_0,R)^q\operatorname{d}\!t\right)^{\frac{1}{q}}
 \right]\left(\frac{|t-\tau|^{\frac{1}{2}}+|x-y|}{R}\right)^\gamma \nonumber\\
&&+C\left[\mathcal {I}^{\mu}_{2-\gamma}(t,x;R)+\mathcal {I}^{\mu}_{2-\gamma}(\tau,y;R)\right]\left(|t-\tau|^{\frac{1}{2}}+|x-y|\right)^\gamma.
\end{eqnarray*}
Therefore, the proof of Theorem \ref{thm:OPE} is complete.
\end{proof}

\section*{Acknowledgments} This work is supported by the National Natural Science Foundation of China (NSFC Grant No.12101452 and No.12071229).

\bibliography{bibliography}

\end{document}